\documentclass[12pt, leqno]{amsart}

\usepackage{float}
\usepackage{graphics}
\usepackage{amssymb,amsmath,amsthm,amscd}
\usepackage{mathrsfs}
\usepackage{enumerate}
\usepackage{color}
\usepackage[vcentermath,enableskew,noautoscale]{youngtab}
\usepackage[all]{xy}
\usepackage{verbatim}

\numberwithin{equation}{section}

\newcommand{\C}{{\mathbf C}}

\newcommand{\Z}{{\mathbf Z}}
\newcommand{\B}{{\mathbf{B}}}
\newcommand{\A}{{\mathbf A}}
\newcommand{\F}{{\mathbf F}}

\newcommand{\V}{{\mathbf V}}

\newcommand{\gl}{{\mathfrak{gl}}}
\newcommand{\q}{{\mathfrak{q}}}

\newcommand{\seteq}{\mathbin{:=}}
\newcommand{\uqqn}{U_q(\mathfrak{q}(n))}
\newcommand{\uqn}{U(\mathfrak{q}(n))}

\theoremstyle{plain}
\newtheorem{lemma}{Lemma}[section]
\newtheorem{prop}[lemma]{Proposition}
\newtheorem{theorem}[lemma]{Theorem}
\newcommand{\Prop}{\begin{prop}}
\newcommand{\enprop}{\end{prop}}
\newcommand{\Lemma}{\begin{lemma}}
\newcommand{\enlemma}{\end{lemma}}
\newcommand{\Th}{\begin{theorem}}
\newcommand{\enth}{\end{theorem}}
\newtheorem{corollary}[lemma]{Corollary}
\newcommand{\Cor}{\begin{corollary}}
\newcommand{\encor}{\end{corollary}}
\newtheorem{definition}[lemma]{Definition}

\newcommand{\Def}{\begin{definition}}
\newcommand{\edf}{\end{definition}}

\theoremstyle{definition}
\newtheorem{remark}[lemma]{Remark}
\newtheorem{example}[lemma]{Example}

\newcommand{\g}{{\mathfrak{g}}}
\newcommand{\Uq}{{U_q(\mathfrak{q}(n))}}
\newcommand{\qn}{{\mathfrak{q}(n)}}

\newcommand{\End}{\operatorname{End}}

\newcommand{\isoto}[1][]{\mathop{\xrightarrow[#1]%
{{\raisebox{-.6ex}[0ex][-.6ex]{$\mspace{2mu}\sim\mspace{2mu}$}}}}}
\newcommand{\tensor}{\otimes}

\newcommand{\eq}{\begin{eqnarray}}
\newcommand{\eneq}{\end{eqnarray}}

\newcommand{\eqn}{\begin{eqnarray*}}
\newcommand{\eneqn}{\end{eqnarray*}}

\newcommand{\on}{\operatorname}
\newcommand{\Ker}{\on{Ker}}

\newcommand{\bni}{\be[{\rm(i)}]}
\newcommand{\bna}{\be[{\rm(a)}]}

\newcommand{\QED}{\end{proof}}
\newcommand{\Proof}{\begin{proof}}

\newcommand{\soplus}{\mathop{\mbox{\normalsize$\bigoplus$}}\limits}

\newcommand{\cl}{\colon}

\newcommand{\id}{\on{id}}
\newcommand{\ba}{\begin{array}}
\newcommand{\ea}{\end{array}}

\newcommand{\bi}{\begin{enumerate}[{\rm(i)}]}

\newcommand{\mono}{\rightarrowtail}

\newcommand{\set}[2]{\left\{#1 \mathbin{;} #2 \right\}}

\newcommand{\Mod}{\operatorname{Mod}}
\newcommand{\SMod}{\operatorname{S-Mod}}
\newcommand{\hs}{\hspace*}

\newcommand{\eqsub}{\begin{subequations}\begin{eqnarray}}
\newcommand{\eneqsub}{\end{eqnarray}\end{subequations}}

\newcommand{\ol}{\overline}

\newcommand{\nc}{\newcommand}
\nc{\la}{\lambda}
\nc{\lam}{\lambda}
\nc{\U}[1][\g]{U_q(#1)}
\nc{\te}{\tilde{e}}
\nc{\tei}{\tilde{e}_i}
\nc{\tf}{\tilde{f}}
\nc{\tfi}{\tilde{f}_i}
\nc{\tU}{\widetilde U_q(\g)}
\nc{\tE}{\tilde{E}}
\nc{\tF}{\tilde{F}}

\nc{\tk}{\tilde{k}}
\nc{\tkone}{\tk_{\ol{1}}}
\nc{\teone}{\tilde{e}_{\ol{1}}}
\nc{\tfone}{\tilde{f}_{\ol{1}}}

\nc{\teibar}{\tilde{e}_{\ol{i}}} \nc{\tfibar}{\tilde{f}_{\ol{i}}}
\nc{\tki}{{\tk}_{\ol {i}}}

\nc{\BZ}{{\mathbb{Z}}}
\nc{\al}{\alpha}
\nc{\qs}{{q}}
\nc{\lan}{\langle}
\nc{\ran}{\rangle}
\nc{\re}{{\mathrm{re}}}
\nc{\wt}{\operatorname{wt}}
\nc{\ch}{\operatorname{ch}}
\nc{\Uf}[1][\g]{U^-_q(#1)}
\nc{\Ue}{U^+_q(\g)}
\nc{\eps}{\varepsilon}
\nc{\vphi}{\varphi}
\nc{\sphi}{\varphi^*}
\nc{\seps}{\varepsilon^*}

\nc{\nn}{\nonumber}
\def\max{{\mathop{\mathrm{max}}}}
\nc{\vp}{\varpi}
\nc{\cls}{{\operatorname{cl}}}
\nc{\Wt}{{\operatorname{Wt}}}
\nc{\Us}{U'_q(\g)}
\nc{\La}{\Lambda}
\nc{\ro}{{\rm(}}
\nc{\rf}{{\rm)}}
\nc{\norm}{{\mathrm{norm}}}
\nc{\qbox}{\quad\mbox}
\nc{\braid}{{\mathfrak{B}}}
\nc{\Ad}{\operatorname{Ad}}
\nc{\Aut}{\operatorname{Aut}}
\nc{\dt}[1]{\tilde{\tilde #1}}
\nc{\Sn}{S^{{\mathrm{norm}}}}
\nc{\aff}{{\mathrm{aff}}}
\nc{\rk}{{\mathrm{rk}}}
\nc{\tQ}{\widetilde{Q}}
\nc{\tP}{\widetilde{P}}
\nc{\tW}{\widetilde{W}}
\nc{\Dyn}{\mathrm{Dyn}}
\nc{\tD}{\widetilde{\Delta}}
\nc{\height}{{\operatorname{ht}}}
\nc{\bl}{\bigl}
\nc{\br}{\bigr}
\nc{\Hecke}{\mathrm{H}}
\nc{\HA}{\Hecke^{\mathrm{A}}}
\nc{\HB}{\Hecke^{\mathrm{B}}}
\nc{\K}{\mathrm{K}}
\newcommand{\scbul}{{\,\raise1pt\hbox{$\scriptscriptstyle\bullet$}\,}}
\nc{\vac}{{\phi}}
\nc{\Bt}{\B_\theta(\g)}
\nc{\be}{\begin{enumerate}}
\nc{\ee}{\end{enumerate}}
\nc{\low}{{\mathrm{low}}}
\nc{\upper}{{\mathrm{up}}}
\nc{\Zodd}{\Z_{\mathrm{odd}}}
\nc{\Ft}[1][n]{\mathbb{P}\mathrm{ol}_{#1}}
\nc{\Ftf}[1][n]{\widetilde{\mathbb{P}\mathrm{ol}}_{#1}}
\nc{\KA}{\on{K}^{\mathrm{A}}}
\nc{\KB}{\on{K}^{\mathrm{B}}}
\nc{\Res}{\on{Res}}
\nc{\Fc}[1][{n,m}]{\mathbf{F}_{#1}}
\nc{\tphi}{\tilde{\varphi}}
\nc{\CO}{\mathscr{O}}
\nc{\inte}{\mathrm{int}}
\nc{\Oint}{\mathcal{O}^{\ge0}_{\inte}}
\nc{\vs}{\vspace}
\nc{\tL}{\widetilde{L}}
\nc{\noi}{\noindent}

\newenvironment{rouge}
{\color{red}}
{}

\nc{\bred}{\begin{rouge}}
\nc{\ered}{\end{rouge}}

\newlength{\mylength}
\title[Crystal bases for the quantum queer superalgebra]
{Crystal bases for the quantum queer superalgebra}

\author[D. Grantcharov, J. H. Jung, S.-J. Kang, M. Kashiwara, M. Kim]{Dimitar Grantcharov$^{1}$, Ji Hye Jung$^{2}$, Seok-Jin
Kang$^{3}$, \\ Masaki Kashiwara$^{4}$, Myungho Kim$^{5}$}

\address{Department of Mathematics \\
         University of Texas at Arlington \\ Arlington, TX 76021, USA}

         \email{grandim@uta.edu}

\address{Department of Mathematical Sciences
         and
         Research Institute of Mathematics \\
         Seoul National University \\ Seoul 151-747, Korea}

         \email{jhjung@math.snu.ac.kr}

\address{Department of Mathematical Sciences
         and
         Research Institute of Mathematics \\
         Seoul National University \\ Seoul 151-747, Korea}

         \email{sjkang@math.snu.ac.kr}

\address{Research Institute for Mathematical Sciences \\
          Kyoto University \\ Kyoto 606-8502, Japan \\
          \& Department of Mathematical Sciences
         and
         Research Institute of Mathematics \\
         Seoul National University \\ Seoul 151-747, Korea}

         \email{masaki@kurims.kyoto-u.ac.jp}

\address{Department of Mathematical Sciences
         and
         Research Institute of Mathematics \\
         Seoul National University \\ Seoul 151-747, Korea}
         \email{mkim@math.snu.ac.kr}

\thanks{$^{1}$This work was partially supported by NSA grant H98230-10-1-0207 and by Max Planck Institute for Mathematics, Bonn.}
\thanks{$^{2}$This work was partially supported by BK21 Mathematical Sciences Division and by NRF Grant \# 2010-0010753.}
\thanks{$^{3}$This work was partially supported by KRF Grant \# 2007-341-C00001, by National Institute for Mathematical Sciences (2010 Thematic Program, TP1004), and by NRF Grant \# 2010-0019516.}
\thanks{$^{4}$This work was partially supported by Grant-in-Aid for Scientific Research (B) 23340005,
Japan Society for the Promotion of Science.}
\thanks{$^{5}$This work was partially supported by KRF Grant \# 2007-341-C00001, and by NRF Grant \# 2010-0019516.}

\keywords{quantum queer superalgebras, crystal bases, odd Kashiwara operators}
\subjclass[2000]{17B37, 81R50}

\begin{document}

\maketitle

\begin{abstract}
In this paper, we develop the crystal basis theory for the quantum
queer superalgebra $\Uq$. We define the notion of crystal bases and
prove the tensor product rule for $\Uq$-modules in the category
$\Oint$. Our main theorem shows that every
$\Uq$-module in the category $\Oint$ has a
unique crystal basis.

\end{abstract}

\section*{Introduction}

For the past 30 years, one of the most striking and influential
developments in combinatorial representation theory would be the
discovery of crystal bases for quantum groups and their
representations \cite{Kas90, Kas91}. Right after its discovery,
the crystal basis theory has attracted a lot of attention and
research activities because it has simple and explicit
combinatorial features and have many significant applications to a
wide variety of mathematical and physical theories. In particular,
crystal bases have an extremely nice behavior with respect to
tensor products, which leads to natural and exciting connections
with combinatorics of Young tableaux and Young walls
(\cite{Ka2003, KK08, KN, MM, N93}). Moreover, inspired by the
original works \cite{Kas90, Kas91, Kas93, Kas94}, many important
and deep results have been established for crystal bases for
quantum groups associated with symmetrizable Kac-Moody algebras
(see, for example, \cite{HK2002, JMMO, KMN1, KMN2, KS97, Li1, Li2,
Naka02, Sai}).
In \cite{Lus90, Lus91}, Lusztig provided a geometric approach to this subject.

On the other hand, not much has been known about crystal bases for
quantum groups corresponding to Lie superalgebras. A major difficulty
one encounters in the superalgebra case is the fact that the
category of finite-dimensional representations is in general not
semisimple. Fortunately, there is an interesting and natural
category of finite-dimensional $U_q(\mathfrak{g})$-modules which is
semisimple for the two super-analogues of the general linear Lie
algebra $\mathfrak{gl} (n)$: $\mathfrak{g}  = \mathfrak{gl} (m | n)$
and $\mathfrak{g} = \mathfrak{q} (n)$. This is the category
$\Oint$ of so-called {\it tensor modules};
i.e., those that appear as submodules of tensor powers ${\bf
V}^{\otimes N}$ of the natural $U_q(\mathfrak{g})$-module ${\bf V}$.
The semisimplicity
of $\Oint$ is verified  in
\cite{BKK} for the general linear Lie superalgebra $\mathfrak{g} =
\mathfrak{gl} (m|n)$ and in \cite{GJKK} for the queer Lie
superalgebra $\mathfrak{g} = \mathfrak{q}(n)$.

Furthermore, the
crystal basis theory of $\Oint$ for
$\mathfrak{g}= \mathfrak{gl} (m |n)$ was developed in \cite{BKK},
while the foundations of the highest weight representation theory of
$U_q (\qn)$ have been established in \cite{GJKK}.

In this paper, we develop the crystal basis theory for $\Uq$-modules
in the category $\Oint$. The (quantum) queer
superalgebra is interesting not only as the remaining case for which
$\Oint$ is semisimple, but also due to its
remarkable combinatorial properties. An example of such properties is
the queer analogue of the celebrated {\it Schur-Weyl duality}, often
referred to as {\it Schur-Weyl-Sergeev duality}, which was obtained
in \cite{Ser} for $U (\qn)$ and in \cite{Ol} for $U_q (\qn)$.

Being very interesting on the one hand, the representation theory of
(quantum) queer superalgebra faces numerous challenges on the other.
The queer Lie superalgebra is the only classical Lie superalgebra
whose Cartan subsuperalgebra  has a nontrivial odd part. As a
result, the highest weight space of any finite-dimensional
$\qn$-module has a structure of a Clifford module and the
corresponding  $\mathfrak{gl} (n)$-module appears with multiplicity
higher than one (in fact, a power of $2$).  Also, as observed in
\cite{GJKK}, due to the different classification of Clifford modules
over $\C$ and $\C (q)$,  the classical limit of an irreducible
highest weight $U_q (\qn)$-module is an irreducible highest weight
$U(\qn)$-module or a direct sum of two irreducible highest weight
$U(\qn)$-modules.  On top of these and in contrast to the case of
$\mathfrak{gl} (m|n)$, the odd root generators $e_{\ol{i}}$ and
$f_{\ol{i}}$ of $U_q (\qn)$ are not nilpotent.

We overcome the  challenges described  above in several steps.
First, we set the ground field to be the field $\C ((q))$ of formal
Laurent power series. By enlarging the base field, we obtain an
equivalence of the two categories of Clifford modules, and in
particular, establish a standard version of the classical limit
theorem. As the next step, we introduce the  {\it odd Kashiwara
operators} $\tilde{e}_{\ol{1}}$, $\tilde{f}_{\ol{1}}$, and
$\tilde{k}_{\ol{1}}$, where $\tilde{k}_{\ol{1}}$ corresponds to an
odd element in the Cartan subsuperalgebra of $\qn$. The definitions
of $\tilde{e}_{\ol{1}}$, $\tilde{f}_{\ol{1}}$ are new in the sense
that they are based solely on the comultiplication formulas for
$e_{\ol 1}$, $f_{\ol{1}}$ and lead to nilpotent operators on $L/qL$,
where $L$ is a crystal lattice. Furthermore, from these definitions,
we deduce a special {\it tensor product rule} for odd Kashiwara
operators.

Our definition of a {\it crystal basis} for a $U_q(\qn)$-module $M$
in the category $\Oint$ is also new: such a
basis is a triple $(L,B,(l_b)_{b \in B})$, where the crystal lattice
$L$ is a free $\C[[q]]$-submodule of $M$, $B$ is a finite
$\mathfrak{gl} (n)$-crystal,  $(l_b)_{b \in B}$ is a family of
nonzero vector spaces such that $L / qL = \soplus_{b \in B} l_b$,
with a set of compatibility conditions for the action of the
Kashiwara operators imposed in addition. The definition of crystal
bases leads naturally to the notion of  {\it abstract
$\qn$-crystals}, an example of which is the
$\mathfrak{gl}(n)$-crystal $B$ in any crystal basis $(L,B,(l_b)_{b
\in B})$. The modified notion of crystals allows us to consider the
multiple occurrence of $\mathfrak{gl} (n)$-crystals corresponding to
a highest weight $U_q(\qn)$-module $M$ in $\Oint$  as a single $\qn$-crystal.

As a result of this new setting,  the existence and uniqueness
theorem for crystal bases is proved for any highest weight (not
necessarily irreducible) module $M$ in the category $\Oint$.
Moreover, the $\qn$-crystal $B$ of $M$ depends
only on the highest weight $\lambda$ of $M$ and hence we may write
$B = B(\lambda)$. In addition to the existence and uniqueness
theorem, the decompositions of the module ${\bf V} \otimes M$ and
the crystal ${\bf B} \otimes B(\lambda)$ are established, where
${\bf B}$ is the crystal of ${\bf V}$. These decompositions are
parametrized by the set of all $\lambda + \varepsilon_j$ such that
$\lambda + \varepsilon_j$ is a strict partition $(j=1, \ldots, n)$.
One of key ingredients of the proof of our main theorem is the
characterization of highest weight vectors in ${\bf B} \otimes
B(\lambda)$ in terms of even Kashiwara operators and the highest
weight vector of $B(\lambda)$. All these statements are verified
simultaneously by a series of interlocking inductive arguments.


This paper is organized as follows. In Section 1, we recall some of
the  basic properties of $U_q(\qn)$-modules in the category
$\Oint$. Section 2 is devoted to the
definitions, examples, and some preparatory statements  related to
crystal bases. In particular, we prove the tensor product rule. In
Section 3, we give algebraic and combinatorial characterizations of
highest weight vectors in $\B^{\otimes N}$.  In Section 4, we prove
our main result: the existence and uniqueness theorem for crystal
bases.

\section{The quantum queer superalgebra}\label{sec:qn}

Let $\F=\C((q))$ be the field of formal
Laurent series in an indeterminate $q$
and let $\A=\C[[q]]$ be the subring of $\F$
consisting of formal power series in $q$. For $k \in \Z_{\ge 0}$, we
define
$$[k]= \frac{q^k - q^{-k}}{q - q^{-1}}, \quad [0]!=1, \quad [k]! =
[k] [k-1] \cdots [2][1].$$

 For an integer $n \geq 2$,
let $P^{\vee} = \Z k_1 \oplus \cdots \oplus \Z k_n$ be a free
abelian group of rank $n$ and let ${\mathfrak h} = \C \otimes_{\Z}
P^{\vee}$ be its complexification. Define the linear functionals
$\epsilon_i \in \mathfrak{h}^*$ by $\epsilon_i(k_j) = \delta_{ij}$
$(i,j=1, \ldots, n)$ and set $P= \Z \epsilon_1 \oplus \cdots \oplus
\Z \epsilon_n$. We denote by $\alpha_i = \epsilon_i -
\epsilon_{i+1}$ the {\em simple roots} and by $h_i=k_i-k_{i+1}$
the {\em simple coroots}.

\Def The {\em quantum queer superalgebra $U_q(\mathfrak{q}(n))$} is
the superalgebra over $\F$ with $1$ generated by the symbols $e_i$,
$f_i$, $e_{\ol i}$, $f_{\ol i}$ $(i=1, \ldots, n-1)$, $q^{h}$ $(h\in
P^\vee)$, $k_{\ol j}$ $(j=1, \ldots, n)$ with the following defining
relations.
\begin{align}
\allowdisplaybreaks
\nonumber & q^{0}=1, \ \ q^{h_1} q^{h_2} = q^{h_1 + h_2} \ \ (h_1,
h_2 \in P^{\vee}), \\
\nonumber & q^h e_i q^{-h} = q^{\alpha_i(h)} e_i \ \ (h\in P^{\vee}), \displaybreak[1]\\
\nonumber & q^h f_i q^{-h} = q^{-\alpha_i(h)} f_i \ \ (h\in P^{\vee}), \displaybreak[1]\\
\nonumber & q^h k_{\ol j} = k_{\ol j} q^h, \displaybreak[1]\\
\nonumber & e_i f_j - f_j e_i = \delta_{ij} \dfrac{q^{k_i - k_{i+1}} - q^{-k_i
+ k_{i+1}}}{q-q^{-1}}, \displaybreak[1]\\
\nonumber & e_i e_j - e_j e_i = f_i f_j - f_j f_i = 0 \quad \text{if} \ |i-j|>1, \displaybreak[1]\\
\nonumber & e_i^2 e_j -(q+q^{-1}) e_i e_j e_i  + e_j e_i^2= 0  \quad \text{if} \ |i-j|=1,\displaybreak[1]\\
\nonumber & f_i^2 f_j - (q+q^{-1}) f_i f_j f_i + f_j f_i^2 = 0  \quad \text{if} \ |i-j|=1,\displaybreak[1]\\
\nonumber & k_{\ol i}^2 = \dfrac{q^{2k_i} - q^{-2k_i}}{q^2 - q^{-2}}, \\
          & k_{\ol i} k_{\ol j} + k_{\ol j} k_{\ol i} =0 \ \ \text{if} \  i \neq j,
\label{eq:defining relations} \\
\nonumber & k_{\ol i} e_i - q e_i k_{\ol i} = e_{\ol i} q^{-k_i}, \ q k_{\ol i}e_{i-1}- e_{i-1}k_{\ol i}=-q^{-k_i} e_{\ol{i-1}}, \displaybreak[1]\\
\nn & k_{\ol i}e_j-e_jk_{\ol i}=0 \quad \text{if} \ j \neq i,i-1, \displaybreak[1]\\
\nn & k_{\ol i} f_i - q f_i k_{\ol i} = -f_{\ol i} q^{k_i}, \ q k_{\ol i} f_{i-1}-f_{i-1}k_{\ol i}=q^{k_i}f_{\ol{i-1}}, \displaybreak[1]\\
\nn & k_{\ol i}f_j-f_jk_{\ol i}=0 \quad \text{if} \ j \neq i, i-1,
\displaybreak[1]  \\ \nonumber & e_i f_{\ol j} - f_{\ol j} e_i =
\delta_{ij} (k_{\ol i}
q^{-k_{i+1}} - k_{\ol{i+1}} q^{-k_i}), \\
\nonumber & e_{\ol i} f_j - f_j e_{\ol i} = \delta_{ij} (k_{\ol i}
q^{k_{i+1}} - k_{\ol{i+1}} q^{k_i}), \\
\nonumber &e_i e_{\ol i} - e_{\ol i} e_i = f_i f_{\ol i} - f_{\ol i} f_i = 0, \\
\nonumber &e_i e_{i+1} - q e_{i+1}e_i =
e_{\overline{i}}e_{\overline{i+1}}+ q
e_{\overline{i+1}}e_{\overline{i}}, \\
\nonumber&q f_{i+1}f_i - f_i f_{i+1} =
f_{\overline{i}}f_{\overline{i+1}}+ q f_{\overline{i+1}}f_{\overline{i}}, \\
\nonumber & e_i^2 e_{\overline{j}} - (q+q^{-1})e_i e_{\overline{j}}
e_i + e_{\overline{j}} e_i^2= 0 \quad \text{if} \ |i-j|=1, \\
\nonumber & f_i^2 f_{\overline{j}} - (q+q^{-1})f_i f_{\overline{j}}
f_i + f_{\overline{j}} f_i^2=0 \quad \text{if} \ |i-j|=1.
\end{align}
\edf

The generators $e_i$, $f_i$ $(i=1, \ldots, n-1)$, $q^{h}$ ($h\in
P^\vee$) are regarded as {\em even} and $e_{\ol i}$, $f_{\ol i}$
$(i=1, \ldots, n-1)$, $k_{\ol j}$ $(j=1, \ldots, n)$ are {\em
odd}. From the defining relations, it is easy to see that the even
generators together with $k_{\ol 1}$ generate the whole algebra
$\Uq$.

\begin{remark}
The generators in \eqref{eq:defining relations}
are different from those in \cite[Theorem 2.1]{GJKK}.
The elements $e_i$, $f_i$, $e_{\ol i}$ and $f_{\ol i}$ in \eqref{eq:defining relations}
 correspond to $q^{k_{i+1}} e_i $, $f_i q^{-k_{i+1}}$, $q^{k_{i+1}}e_{\ol i}$
 and $f_{\ol i}q^{-k_{i+1}}$
 in \cite[Theorem 2.1]{GJKK}, respectively.
We rewrite the whole defining relations in \cite[Theorem 2.1]{GJKK}
 in terms of new generators and remove some relations which can be derived from the others.
\end{remark}

The superalgebra $U_q(\mathfrak{q}(n))$ is a bialgebra with the
comultiplication $\Delta\cl U_q(\mathfrak{q}(n)) \to
U_q(\mathfrak{q}(n)) \otimes U_q(\mathfrak{q}(n))$ defined by
\begin{equation}
\begin{aligned}
& \Delta(q^{h})  = q^{h} \otimes q^{h}\quad\text{for $h\in P^\vee$,} \\
& \Delta(e_i)  = e_i \otimes q^{-k_i + k_{i+1}} + 1 \otimes e_i, \\
& \Delta(f_i)  = f_i \otimes 1 + q^{k_i - k_{i+1}} \otimes f_i, \\
& \Delta(k_{\ol 1}) =k_{\ol 1}\otimes q^{k_1}+ q^{-k_1} \otimes k_{\ol 1}.
\end{aligned}
\end{equation}

Let $U^{+}$ (resp.\ $U^{-}$) be the subalgebra of
$U_q(\mathfrak{q}(n))$ generated by $e_i$, $e_{\ol i}$
$(i=1,\ldots, n-1)$ (resp.\ $f_i$, $f_{\ol i}$ ($i=1, \ldots, n-1$)), and
let $U^{0}$ be the subalgebra
generated by $q^{h}$ ($h\in P^\vee$) and $k_{\ol j}$ $(j=1, \ldots, n)$.
In \cite{GJKK}, it was
shown that the algebra $U_q(\mathfrak{q}(n))$ has the {\em
triangular decomposition}:
\begin{equation}
U^{-} \otimes U^{0} \otimes U^{+}\isoto
U_q(\mathfrak{q}(n)).
\end{equation}

 Hereafter, a $U_q(\mathfrak{q}(n))$-module is understood as a
$U_q(\mathfrak{q}(n))$-supermodule.
A $U_q(\mathfrak{q}(n))$-module $M$ is called a {\em weight module}
if $M$ has a weight space decomposition $M=\soplus_{\mu \in P}
M_{\mu}$, where
$$M_{\mu}\seteq  \set{ m \in M}{q^h m = q^{\mu(h)} m \ \ \text{for all} \ h
\in P^{\vee} }.$$ The set of weights of $M$ is defined to be
$$\wt(M) = \set{\mu \in P}{M_{\mu} \neq 0 }.$$

\Def A weight module $V$ is called a {\em highest weight module
with highest weight $\la \in P$} if $V_{\la}$ is finite-dimensional and
satisfies the following conditions:
\bna
\item $V$ is generated by $V_{\la}$,
\item $e_i v = e_{\ol i} v =0$ for all $v \in V_{\la}$,
$i=1, \ldots, n-1$.
\ee  \edf

As seen in \cite{GJKK}, there exists a unique irreducible highest weight module with highest
weight $\la \in P$ up to parity change, which will be denoted by
$V(\la)$.

Set
\begin{equation*}
\begin{aligned}
P^{\ge 0} = & \{ \la = \la_1 \epsilon_1 + \cdots + \la_n \epsilon_n
\in P\, ; \, \la_j \in \Z_{\ge 0} \ \ \text{for all} \ j=1, \ldots, n \}, \\
\La^{+} = & \{\la = \la_1 \epsilon_1 + \cdots + \la_n \epsilon_n \in
P^{\ge 0}\, ; \, \text{$\la_{i} \ge \la_{i+1}$ and $\la_{i}=\la_{i+1}$ implies}\\
 &  \hs{31ex}\text{$\la_{i} = \la_{i+1} = 0$ for all $i=1,
 \ldots,n-1$}\}.
\end{aligned}
\end{equation*}
Note that each element $\la \in \La^{+}$ corresponds to a {\em
strict partition} $\la = (\la_1 > \la_2 > \cdots > \la_r >0)$.
Thus we will often call $\la \in \La^{+}$ a strict partition. With
the same reason, we call $\lambda=(\la_1,\la_2,\ldots,\la_n) \in
P^{\geq 0}$ a {\it partition} if $\la_1 \geq \la_2 \geq \cdots
\geq \la_r > \la_{r+1}=\cdots =\la_n=0$. We denote $r$  by
$\ell(\la)$.

\begin{example}
Let $$\V = \soplus_{j=1}^n \F v_{j} \oplus \soplus_{j=1}^n \F
v_{\ol j}$$ be the vector representation of $U_q(\mathfrak{q}(n))$.
The action of $\Uq$ on $\V$ is given as follows:
\begin{equation}
\ba{llll}
e_iv_j=\delta_{j,i+1}v_i, &e_iv_{\ol j}=\delta_{j,i+1}v_{\ol i},
&f_iv_j=\delta_{j,i}v_{i+1},&f_iv_{\ol j}=\delta_{j,i}v_{\ol{i+1}}, \\[1ex]
 e_{\ol i}v_j=\delta_{j,i+1}v_{\ol{i}},&e_{\ol i}v_{\ol j}=\delta_{j,i+1}v_{i},&
f_{\ol i}v_j=\delta_{j,i}v_{\ol{i+1}},&f_{\ol i}v_{\ol j}=\delta_{j,i}v_{{i+1}}, \\[1ex]
q^h v_j=q^{\epsilon_j(h)} v_j, &q^h v_{\ol j}=q^{\epsilon_j(h)} v_{\ol j},
&k_{\ol i}v_j=\delta_{j,i}v_{\ol j},&k_{\ol i}v_{\ol j}=\delta_{j,i}v_{j}.
\ea
\end{equation}
Note that $\V$ is an irreducible highest weight module with highest weight $\epsilon_1$.
 \end{example}

\Def We define $\Oint$ to be the category of
finite-dimensional weight modules $M$ satisfying the following conditions:
\bna
\item  $\wt(M) \subset P^{\ge 0}$,
\item for any $\mu\in P^{\ge0}$ and $i \in \{1,\ldots, n\}$
 such that $\lan k_i,\mu\ran=0$,
we have $k_{\ol i}\vert_{M_\mu}=0$.
\ee
\edf

\begin{remark} 
By Lemma~\ref{le:invarianct under tildeki} below,
it is enough to assume $i=1$ in the
condition (b).
Note also that the condition (b) is equivalent to saying that
every weight space $M_\mu$ is completely reducible as a $U^0$-module.
\end{remark}

The fundamental properties of the category $\Oint$ are summarized in the following proposition.

\vskip 2ex

\Prop [\cite{GJKK}] \hfill

\bna

\item Every $U_q(\mathfrak{q}(n))$-module in
$\Oint$ is completely reducible.

\item Every irreducible object in $\Oint$ has the
form $V(\la)$ for some $\la \in \La^{+}$.

\item The category $\Oint$ is stable under
tensor products.
\ee
\enprop

In \cite{GJKK}, we employed the rational function field $\C(q)$
as the base field of $\uqqn$. 
But here, we  employ $\C((q))$ instead of $\C(q)$ as the base field of
$\uqqn$. 
Note that when $m$ is a non-negative integer, the
$q$-integer $\dfrac{q^{2m}-q^{-2m}}{q^2-q^{-2}}$  has a square root
in $\C((q))$ but not in $\C(q)$.
This difference gives the following two statements,
which is simpler than the corresponding statements
in \cite{GJKK}.

\begin{prop}[{cf.\ \cite[Corollary 3.9]{GJKK})}]
Let ${\rm Cliff} _q(\la)$ be the associative superalgebra over
$\C((q))$ generated by odd generators
$\set{t_{\ol i}}{i=1,2,\ldots, n}$ with the
defining relations
$$
\begin{array}{cc}
t_{\ol i} t_{\ol j} + t_{\ol j} t_{\ol i} = \delta_{ij}
\dfrac{2(q^{2\la_i}-q^{-2\la_i})}{q^2-q^{-2}}, & i,j =
1,2,\ldots,n.
\end{array}
$$
Then ${\rm Cliff} _q(\la)$ 
has up to isomorphism
\bna
\item two simple modules $E^q(\la)$ and $\Pi(E^q(\la))$ of dimension $2^{k-1} | 2^{k-1}$ if $\ell(\la)=2k$,
\item one simple module $E^q(\la) \cong \Pi(E^q(\la))$ of dimension $2^k | 2^k$ if $\ell(\la)=2k+1$.
\end{enumerate}
\end{prop}

\begin{prop}[{cf.\ \cite[Theorem 5.14]{GJKK}}]
 \label{prop:classical limit}
Let $V(\la)$ be an irreducible highest weight module with highest weight $\la \in \Lambda^+$.
Then we have
$$\ch V(\la) = \ch V_{\cls}(\la),$$
where $V_{\cls}(\la)$ is an irreducible highest weight module over $\q(n)$ with highest weight $\la$.
\end{prop}

In short, contrary to \cite{GJKK}, we have the same classification for
the modules over ${\rm Cliff} _q(\la)$ as that for the modules over
the Clifford algebra with the base field $\C$. Also we have the same characters
of the irreducible modules over $\uqqn$ as those of the
irreducible modules over $\q(n)$.

\begin{remark} \label{rem:Grothendieck rings}
Define $\mathcal O^{\geq 0}_{\inte,\cls}$ to be the category of finite-dimensional weight modules $M$ over $\q(n)$
such that i) $\wt(M) \subset P^{\geq 0}$,
ii) $k_{\ol i}|_{M_\mu}=0$ for $i \in \{1, \ldots, n \}$ and $\mu \in P^{\geq 0}$
satisfying $\langle k_i, \mu \rangle =0$.
Here $k_{\ol i}$ is the element of $\q(n)$ given by
$\left(
\begin{array}{cc} 0& E_{i,i} \\ E_{i,i} & 0
\end{array} \right) $, where $E_{i,i}$ is the $n \times n$-matrix having $1$ in the $(i,i)$-position and $0$ elsewhere.
Let us denote the Grothendieck rings of the categories by
$K(\mathcal O^{\geq 0}_{\inte})$ and $K(\mathcal O^{\geq 0}_{\inte,\cls})$, respectively.
Since $\mathcal O^{\geq 0}_{\inte,\cls}$
and $\Oint$ are semisimple categories,
by taking the classical limit (i.e., taking the reduction at $q=1$),
we have a ring isomorphism
$$K(\mathcal O^{\geq 0}_{\inte}) \isoto K(\mathcal O^{\geq 0}_{\inte,\cls})$$
which sends $V(\la) \mapsto V_{\cls}(\la)$. 
\end{remark}

Now we give a decomposition of the tensor product of the natural
representation with a highest weight module.

\Th\label{th:decomposition} Let $M$ 
be a  highest weight $\Uq$-module in $\Oint$ with
highest weight $\lambda \in \Lambda^+$. Then we have
$$\V \otimes M \simeq\soplus_{\stackrel{\la + \epsilon_j :}{ \text{strict partition}}}
M_ j,$$ where $M_j$ is a highest weight $\Uq$-module in the
category $\Oint$ with highest weight $\la +
\epsilon_j$ and $\dim (M_{j})_{\la + \epsilon_j} = 2 \dim
M_{\la}$.
\enth
\begin{proof}
We will prove that our assertion holds for finite-dimensional
highest weight modules over $\q(n)$. Then,
by Remark~\ref{rem:Grothendieck rings}, our assertion holds also for
finite-dimensional highest weight modules over $\uqqn$.

Let $\uqn$ be the universal enveloping algebra of $\q(n)$ and
let $U^{\geq 0}$ be the universal enveloping algebra of the standard
Borel subalgebra of $\q(n)$. Let $M$ be a highest weight
$\uqn$-module with highest weight $\la \in \Lambda^+$ and $\V_{\cls}
= \soplus_{i=1}^n (\C v_i \oplus \C v_{\ol i})$ be the natural
representation of $\uqn$. Consider a surjective homomorphism
$$\uqn \tensor_{U^{\geq 0}} \mathbf v_\la \twoheadrightarrow M,$$
where $\mathbf v_\la \simeq M_\la$ as a $U^{\geq 0}$-module.
Now we have
$$ \V_{\cls} \tensor \bigl(\uqn \tensor_{U^{\geq 0}} \mathbf v_\la\bigr) \simeq
 \uqn \tensor_{U^{\geq 0}}  (\V_{\cls} \tensor \mathbf v_\la).
$$
Then $
F_i(\V_{\cls} \tensor \mathbf v_\la) \seteq
\soplus_{j \leq i}(\C v_j \oplus \C v_{\ol j}) \tensor \mathbf v_\la$
is a $U^{\geq 0}$-module. We set
$$
N \seteq \uqn \tensor_{U^{\geq 0}}  (\V_{\cls} \tensor \mathbf v_\la),
\quad
F_i(N) \seteq \uqn \tensor_{U^{\geq 0}} F_i(\V_{\cls} \otimes \mathbf v_\la).
$$
Since
$$
F_i(\V_{\cls} \tensor \mathbf v_\la) / F_{i-1}(\V_{\cls} \tensor
\mathbf v_\la) \simeq (\C v_i \oplus \C v_{\ol i}) \tensor \mathbf
v_\la,
$$
we see that
$$
F_i(N) / F_{i-1}(N) \simeq
\uqn \tensor_{U^{\geq 0}} \big(F_i(\V_{\cls} \tensor \mathbf v_\la) /
F_{i-1}(\V_{\cls} \tensor \mathbf v_\la)\big)
$$
is a highest weight module with highest weight $\la + \eps_i$.

\medskip
Now we shall show
\eq
N\simeq\soplus_{k \leq r} \big( F_k(N) / F_{k-1}(N)\big)\oplus
N / F_r(N),\quad \text{where $r=\ell(\la)$.}
\label{eq:fildec}
\eneq
First note that $F_i(N) / F_{i-1}(N)$ admits the central character
$$\chi_i \seteq \chi_{\la+\eps_i}\cl\mathcal Z \to \C,$$
where $\mathcal Z$ is the center of $\uqn$ and $\chi_\mu$ is the
central character afforded by the Weyl module $W(\mu)$ with
highest weight $\mu$ (see \cite[Section 1]{GJKK} for Weyl modules
and central characters). From \cite[Proposition 1.7]{GJKK}, we
know that $\chi_1,\ldots, \chi_r, \chi_{r+1}$ are different from
each other, and $\chi_{r+1} = \chi_{r+2} = \cdots = \chi_n$.

\noindent
Let us choose an element  $a \in \mathcal Z$ such that $\chi_1(a)
= \cdots = \chi_r(a)=0$ and $\chi_{r+1}(a) \neq 0$. Then we have
$ a |_{F_i(N)/F_{i-1}(N)} =0$ and hence $a F_i(N) \subset
F_{i-1}(N)$ for $i \leq r$. It follows that $a^r F_r(N) =F_{-1}(N)
=0$. Hence $N \stackrel{a^r}{\rightarrow} N$ factors through $N
\rightarrow N/F_r(N) \stackrel{\psi}{\rightarrow} N$. Since $a^r\cl
N/F_r(N) \to N/F_r(N)$ is an isomorphism, we have the diagram
\begin{displaymath}
\xymatrix@R=2ex {
&N \ar[rd]\\
N/F_r(N) \ar[ru]^-{\psi} \ar[rr]_{a^r}^-{\sim} & &  N/F_r(N)\,.}
\end{displaymath}

It follows that
$$N \simeq (N/F_r(N)) \oplus F_r(N).$$
Using a similar argument, we can conclude that
$$F_k(N) \simeq (F_k(N)/F_{k-1}(N)) \oplus F_{k-1}(N)$$
for $k \leq r$. Hence we obtain \eqref{eq:fildec}.

By \cite[Proposition 1.4 (3)]{GJKK}, we know that $F_i(N) /
F_{i-1}(N)$ admits a finite-dimensional quotient
if and only if $\la+\eps_i$ is a strict partition,
and $N/F_r(N)$ has only trivial finite-dimensional quotient.
Since $\V_{\cls} \tensor M$
is a largest finite-dimensional quotient of $N$, we get the desired result.
\end{proof}

\begin{corollary}\label{cor:Vtens}
Any irreducible $\uqqn$-module in
$\Oint$ appears as a direct summand of tensor products of\/ $\V$.
\end{corollary}
\Proof
It follows immediately
from Theorem~\ref{th:decomposition}.
\QED

\section{Crystal bases}

Let $M$ be a $U_q(\mathfrak{q}(n))$-module in the category
$\Oint$. For $i=1, 2, \ldots, n-1$, let $u \in
M_{\la}$ $(\la \in P)$ be a weight vector and consider the {\em
$i$-string decomposition} of $u$:
$$u =\sum_{k\ge 0} f_i^{(k)} u_k,$$
where $e_i u_k =0$ for all $k \ge 0$ and $f_i^{(k)} = f_i^{k} /
[k]!$. We define the {\em even Kashiwara operators} $\tei$, $\tfi$
$(i=1, \ldots, n-1)$ by
\begin{equation}
\begin{aligned}
& \tei u = \sum_{k \ge 1} f_i^{(k-1)} u_k, \qquad
 \tfi u = \sum_{k \ge 0} f_i^{(k+1)} u_k.
\end{aligned}
\end{equation}
On the other hand, we define the {\em odd Kashiwara operators}
$\tilde{k}_{\ol {1}}$, $\tilde{e}_{\ol {1}}$, $\tilde{f}_{\ol
{1}}$ by
\begin{equation}
\begin{aligned}
\tkone & = q^{k_1-1}k_{\ol 1}, \\
\teone & = - (e_1 k_{\ol 1} - q k_{\ol 1} e_1) q^{k_1 -1}, \\
\tfone & = - (k_{\ol 1} f_1 - q f_1 k_{\ol 1}) q^{k_2-1}.
\end{aligned}
\end{equation}

The following lemma is obvious.
\begin{lemma}\label{com:evenodd}
The operators $\teone$ and $\tfone$ commute with
$\te_i$ and $\tf_i$ $(3\le i\le n-1)$.
\end{lemma}

Recall that an {\em abstract $\mathfrak{gl}(n)$-crystal} is a set
$B$ together with the maps $\tei, \tfi\cl B \to B \sqcup \{0\}$,
$\vphi_i, \eps_i \cl B \to \Z \sqcup \{-\infty\}$ $(i \in I =\{1,
\ldots, n-1\})$, and $\wt\cl B \to P$ satisfying the following
conditions (see \cite{Kas93}):

\begin{itemize}
\item[(i)] $\wt(\tei b) = \wt b + \alpha_i$ if $i\in I$ and
$\tei b \neq 0$,

\item[(ii)] $\wt(\tfi b) = \wt b - \alpha_i$ if $i\in I$ and
$\tfi b \neq 0$,

\item[(iii)] for any $i \in I$ and $b\in B$, $\vphi_i(b) = \eps_i(b) +
\langle h_i, \wt b \rangle$,

\item[(iv)] for any $i\in I$ and $b,b'\in B$,
$\tfi b = b'$ if and only if $b = \tei b'$,

\item[(v)] for any $i \in I$ and $b\in B$
such that $\tei b \neq 0$, we have  $\eps_i(\tei b) = \eps_i(b) -
1$, $\vphi_i(\tei b) = \vphi_i(b) + 1$,

\item[(vi)] for any $i \in I$ and $b\in B$ such that $\tfi b \neq 0$,
we have $\eps_i(\tfi b) = \eps_i(b) + 1$, $\vphi_i(\tfi b) =
\vphi_i(b) - 1$,

\item[(vii)] for any $i \in I$ and $b\in B$ such that $\vphi_i(b) = -\infty$, we
have $\tei b = \tfi b = 0$.

\end{itemize}

In this paper, we say that an abstract $\mathfrak{gl}(n)$-crystal is
a {\em $\mathfrak{gl}(n)$-crystal} if it is realized as a crystal
basis of a finite-dimensional integrable
$U_q(\mathfrak{gl}(n))$-module. In particular, for any $b$ in a
$\mathfrak{gl}(n)$-crystal $B$, we have
$$\eps_i(b)=\max\{n\in\Z_{\ge0}\,;\,\tei^nb\not=0\}, \quad
\vphi_i(b)=\max\{n\in\Z_{\ge0}\,;\,\tfi^nb\not=0\}.$$

\Def \label{def:crystal base}
Let $M= \soplus_{\mu \in P^{\ge 0}} M_{\mu}$ be a
$U_q(\mathfrak{q}(n))$-module in the category
$\Oint$. A {\em crystal basis} of $M$ is a
triple $(L, B, l_{B}=(l_{b})_{b\in B})$, where
\bna
\item $L$ is a free $\A$-submodule of $M$ such that

\bni
\item $\F \otimes_{\A} L \isoto M$,

\item $L = \soplus_{\mu \in P^{\ge 0}} L_{\mu}$, where $L_{\mu} = L
\cap M_{\mu}$,

\item  $L$ is stable under the Kashiwara operators $\tei$,
$\tfi$ $(i=1, \ldots, n-1)$, $\tkone$, $\teone$, $\tfone$.
\end{enumerate}

\item $B$ is a $\mathfrak{gl}(n)$-crystal together with
the maps $\teone, \tfone \cl B \to B \sqcup \{0\}$ such that

\bni
\item $\wt(\teone b) = \wt(b) + \alpha_1$, $\wt(\tfone b) = \wt(b) -
\alpha_1$,

\item for all $b, b' \in B$, $\tfone b = b'$ if and only if $b = \teone b'$.
\end{enumerate}

\item $l_{B}=(l_{b})_{b \in B}$ is a family of 
non-zero $\C$-vector
subspaces of $L/qL$ 
such that

\bni
\item $l_{b} \subset (L/qL)_{\mu}$ for $b \in B_{\mu}$,

\item  $L/qL = \soplus_{b \in B} l_{b}$,

\item $\tkone l_{b} \subset l_{b}$,
\item for $i=1, \ldots, n-1, \ol 1$, we have
\be[{\rm(1)}]
\item
if $\tei b=0$ then $\tei l_{b} =0$, and
otherwise $\tei$ induces an isomorphism
$l_{b}\isoto l_{\tei b}$,
\item
if $\tfi b=0$ then $\tfi l_{b} =0$,
and otherwise $\tfi$ induces an isomorphism $l_{b}\isoto l_{\tfi b}$.
\ee
\end{enumerate}
\end{enumerate}

\edf

\begin{prop}
 Let $(L, B, l_{B})$ be a crystal basis of a $\uqqn$-module $M$.
 Then we have
  $$\teone^2 = \tfone^2 = 0$$ as endomorphisms on $L/qL$.
\begin{proof}
  Since every $u \in L_{\la}$ has a $1$-string decomposition
$u=\sum_{k=0}^N f_1^{(k)}u_k$ with $e_1 u_k=0$ for $k=0, \ldots,
N$, it suffices to show that $\teone^2 u \equiv \tfone^2 u \equiv 0$ (mod $qL$) for
$u=f_1^{(s)}v$ with $e_1 v =0$ and $\wt(v) = \mu$ ($s \geq 0$).

We first show $\teone^2 u \equiv 0$ (mod $qL)$. 
From the defining relations $k_{\ol 1} e_1 -q e_1 k_{\ol 1}=e_{\ol
1}q^{-k_1}$ and $e_1 e_{\ol 1}=e_{\ol 1}e_1$, we obtain
$$e_1 k_{\ol 1}e_1-qe_1^2k_{\ol 1}=e_1 e_{\ol 1}q^{-k_1} \quad \text{and} \quad k_{\ol
1}e_1^2-qe_1k_{\ol 1} e_1=q^{-1} e_{\ol 1} e_1 q^{-k_1}=q^{-1} e_1
e_{\ol 1} q^{-k_1}.$$ Then we have
$$ e_1 k_{\ol 1} e_1 -qe_1^2 k_{\ol 1} = q k_{\ol 1}e_1^2 - q^2
e_1 k_{\ol 1} e_1.$$ That is,
\eq
\ e_1 k_{\ol 1} e_1=e_1^{(2)} k_{\ol 1}
+k_{\ol 1 } e_1^{(2)}.\label{eq:k1serre}
\eneq
Using this formula, we obtain
\begin{equation*}
\begin{aligned}
\teone^{2} &=(e_1 k_{\ol 1} -q k_{\ol 1} e_1)^2 q^{2k_1-1} \\
           &=( (e_1^{(2)}
k_{\ol 1} +k_{\ol 1 } e_1^{(2)}) k_{\ol 1} -q e_1 k_{\ol 1}^2 e_1
-q k_{\ol 1}
           e_1^2 k_{\ol 1}+q^2 k_{\ol 1} (e_1^{(2)}
k_{\ol 1} +k_{\ol 1 } e_1^{(2)}))q^{2k_1-1} \\
           &= \dfrac{q-q^{-1}}{q+q^{-1}} q^2 e_1^2 q^{4k_1}.
\end{aligned}
\end{equation*}
It follows that
\begin{equation*}
\begin{aligned}
\teone^2 u & = \dfrac{q-q^{-1}}{q+q^{-1}} q^{\langle 4k_1, \mu -s
\alpha_1 \rangle +2} e_1^2 f_1^{(s)}v \\
&=\dfrac{q-q^{-1}}{q+q^{-1}} q^{4\langle k_1, \mu \rangle -4s +2}
[\langle k_1-k_2, \mu \rangle -s+1 ] [\langle k_1-k_2, \mu \rangle
-s+2]f_1^{(s-2)}v.
\end{aligned}
\end{equation*}
Note that $q^{2\langle k_1-k_2, \mu \rangle -2s +1} [\langle
k_1-k_2, \mu \rangle -s+1 ] [\langle k_1-k_2, \mu \rangle -s+2]
\equiv 1 $ (mod $q \A)$. Since
\eqn
&&4 \langle k_1, \mu \rangle -4s +2 -(2 \langle k_1-k_2, \mu \rangle
-2s +1)\\
&&\hs{10ex} = 2 (\langle k_1-k_2, \mu \rangle -s)+4 \langle k_2, \mu \rangle
+1\ge1,
\eneqn
we have
$$q^{4\langle k_1, \mu \rangle -4s +2}
[\langle k_1-k_2, \mu \rangle -s+1 ] [\langle k_1-k_2, \mu \rangle
-s+2] \in q \A,
$$
which implies $\teone^2 u \equiv 0$ (mod $qL$) as desired.

Now we show $\tfone^2 u \equiv 0$ (mod $qL$). By a similar argument
as above, we obtain
$$ f_1 k_{\ol 1} f_1 = f_1^{(2)}k_{\ol 1} +k_{\ol 1} f_1^{(2)}.$$
Then we have \begin{equation*}
\begin{aligned}
\tfone^{2}&=(k_{\ol 1}f_1 -q f_1 k_{\ol 1})^2 q^{2k_2 -1}\\
          &=( k_{\ol{1}} (f_1^{(2)}k_{\ol 1} +k_{\ol 1} f_1^{(2)})
          -q k_{\ol 1}f_1^2 k_{\ol 1} -q f_1 k_{\ol 1}^2 f_1 +q^2
          (f_1^{(2)}k_{\ol 1} +k_{\ol 1} f_1^{(2)}) k_{\ol 1})q^{2k_2
          -1}\\
          &=\dfrac{q-q^{-1}}{q+q^{-1}}f_1^2q^{2k_1+2k_2-2} .
\end{aligned}
\end{equation*}
It follows that
$$
\begin{array}{ll}
\tfone^2 u&=
\dfrac{q-q^{-1}}{q+q^{-1}}f_1^2 q^{ \langle 2k_1+2k_2, \mu-s \alpha_1 \rangle -2 }f_1^{(s)}v \\
  &=  \dfrac{q-q^{-1}}{q+q^{-1}}     q^{2 \langle k_1+ k_2 , \mu  \rangle-2}  [s+2] [s+1] f_1^{(s+2) } v.
\end{array}
$$
If $ \langle k_1-k_2, \mu \rangle < s+2$, then $f_1^{(s+2)}v=0$,
i.e., $\tfone^2 u \equiv 0 $ (mod $qL$). If $ \langle k_1-k_2, \mu
\rangle  \geq s+2 $, we have
$$ 2\langle k_1+k_2, \mu
\rangle -2 \geq  2 \langle k_1-k_2, \mu \rangle -2 \geq  2s+2.$$
Since $q^{2s+1}[s+2][s+1] \equiv 1$ mod $q\A$, we have
 $$q^{\langle 2k_1 +2k_2 , \mu \rangle -2} [s+2][s+1] \in
q \A,$$ which proves our assertion.
\end{proof}
 \end{prop}

\vskip 3mm

\begin{example}
Let $\V = \soplus_{j=1}^n \F v_{j} \oplus \soplus_{j=1}^n \F
v_{\ol j}$ be the vector representation of $U_q(\mathfrak{q}(n))$.
Set $$\mathbf{L} = \soplus_{j=1}^n \A v_{j} \oplus
\soplus_{j=1}^n \A v_{\ol j}\quad \text{and }l_{j} = \C v_{j} \oplus \C
v_{\ol j} \subset \mathbf{L}/ q \mathbf{L},$$ and let $\B$ be the
$\mathfrak{gl}(n)$-crystal with the $\bar 1$-arrow given below.

$$\xymatrix@C=5ex
{*+{\young(1)} \ar@<0.1ex>[r]^-{1}
\ar@{-->}@<-0.9ex>[r]_{\ol 1} & *+{\young(2)} \ar[r]^2 & *+{\young(3)} \ar[r]^3 & \cdots \ar[r]^{n-1} & *+{\young(n)} }.$$
Here, the actions of $\tfi$ $(i=1, \ldots, n-1, \ol 1)$ are
expressed by $i$-arrows. Then $(\mathbf{L}, \B,
l_{\B}=(l_j)_{j=1}^n)$ is a crystal basis of $\V$.
\end{example}

\begin{remark}  \label{rem:gln structure}
Let $M$ be a $\uqqn$-module in the category  $\Oint$
with a crystal basis $(L,B,l_{B})$,
and let $B=\coprod_{k=1}^s B_k$ be the decomposition of $B$
into 
connected $\mathfrak{gl}(n)$-crystals.
Then there exists a decomposition
$$M = \soplus_{k=1}^{s} \soplus_{j=1}^{m_k} M_{k,j} $$
of $M$ as a $U_q(\mathfrak{gl}(n))$-module, where \bna
\item $m_k=\dim l_b$ for some $b \in B_k$,
\item $M_{k,j}$ has a $U_q(\mathfrak{gl}(n))$-crystal basis $(L_{k,j}, B_{k,j})$ such that
\bni
\item $L=\soplus_{k,j}L_{k,j}$, 
\item there exists a $\mathfrak{gl}(n)$-crystal isomorphism $\phi_{k,j} \cl B_k \isoto B_{k,j}$ so that
the vectors $\phi_{k,j}(b)$ $(j=1, \ldots, m_k)$ form a basis of
$l_b$ for each $b \in B_k$. \ee \ee
\end{remark}

\begin{remark}
Let $M$ be a $\uqqn$-module in the category  $\Oint$
with a crystal basis $(L,B,l_{B})$.
For $i=1,\ldots,n-1,\ol1$ and $b$, $b'\in B$, if $b'=\tf_ib$ is satisfied,
then we have isomorphisms $\tf_i\cl l_b\isoto l_{b'}$ and
$\te_i\cl l_{b'} \isoto l_b$.
If $i=1,\ldots,n-1$, then they are inverses to each other
by Remark~\ref{rem:gln structure}.
However, when $i=\ol1$, they are not inverses to each other in general.
\end{remark}

The {\it tensor product rule} given in the following theorem is one
of the most important features of crystal basis theory.

\Th \label{th2:tensor product}
 Let $M_j$ be a $\uqqn$-module in $\Oint$ with a crystal basis
$(L_j, B_j, l_{B_j})$ $(j=1,2)$. Set $B_1\otimes B_2 = B_1 \times B_2$ and
$l_{b_1\otimes b_2}=l_{b_1} \otimes l_{b_2}$ for $b_1\in B_1$ and $b_2\in B_2$.
Then $(L_1 \otimes_{\A} L_2, B_1 \otimes B_2,(l_b)_{b\in B_1 \otimes B_2})$
is a crystal basis of $M_1 \otimes_{\F}M_2$,
where the action of the Kashiwara operators on $B_1\otimes B_2$
are given as follows:
\eq
&&\begin{aligned}
\tei(b_1 \otimes b_2) & = \begin{cases} \tei b_1 \otimes b_2 \ &
\text{if} \ \vphi_i(b_1) \ge \eps_i(b_2), \\
b_1 \otimes \tei b_2 \ & \text{if} \ \vphi_i(b_1) < \eps_i(b_2),
\end{cases} \\
\tfi(b_1 \otimes b_2) & = \begin{cases} \tfi b_1 \otimes b_2 \
& \text{if} \  \vphi_i(b_1) > \eps_i(b_2), \\
b_1 \otimes \tfi b_2 \ & \text{if} \ \vphi_i(b_1) \le \eps_i(b_2),
\end{cases}
\end{aligned}\\[2ex]
 \label{eq1:tensor product}
&&\begin{aligned}
\teone (b_1 \otimes b_2) & = \begin{cases} \teone b_1 \otimes b_2
& \text{if} \ \lan k_1, \wt b_2 \ran =   \lan k_2, \wt b_2 \ran =0, \\
b_1 \otimes \teone b_2   
&  \text{otherwise,}
\end{cases} \\
\tfone(b_1 \otimes b_2) & = \begin{cases} \tfone b_1 \otimes b_2
& \text{if} \ \lan k_1, \wt b_2 \ran = \lan k_2, \wt b_2 \ran =0, \\
b_1 \otimes \tfone b_2   
& \text{otherwise}.
\end{cases}\\
\end{aligned}
\label{eq2:tensor product}
\eneq
 \enth

\begin{proof}
It is obvious that
\begin{equation*}
\begin{aligned}
& (L_1\tensor L_2)/q(L_1 \tensor L_2) =\soplus_{b_1 \in B_1,  b_2
\in  B_2} l_{b_1} \tensor l_{b_2},\\
& l_{b_1} \tensor l_{b_2} \subset ((L_1 \tensor L_2 )/ q(L_1
\tensor L_2) )_{\la+\mu} \ \ \text{for} \ b_1 \in (B_1)_{\la}, b_2
\in (B_2)_{\mu}.
\end{aligned}
\end{equation*}

For $i=1, 2, \ldots, n-1$, our assertions were already proved in
\cite{Kas90, Kas91}. Let us show the $i= {\ol 1}$ case. The following
comultiplication formulas can be checked easily:
\begin{equation*}
\left\{\begin{aligned}
& \Delta(\tkone) = \tkone \otimes q^{2 k_1} + 1 \otimes \tkone, \\
& \Delta(\teone) = \teone \otimes q^{k_1 + k_2} + 1 \otimes \teone
 - (1-q^2) \tkone \otimes e_1 q^{2 k_1}, \\
& \Delta(\tfone) = \tfone \otimes q^{k_1 + k_2} + 1 \otimes \tfone
 - (1-q^{2}) \tkone \otimes f_1 q^{k_1 + k_2-1}.
\end{aligned}
\right.
\end{equation*}
Clearly, $L_1 \tensor L_2$ and $l_{b_1} \tensor l_{b_2}$ are
stable under $\Delta (\tkone)$ for all $b_1 \in B_1, b_2 \in B_2$.

We will show that $L_1 \otimes L_2$ is stable under
$\Delta(\teone)$ and $\Delta(\tfone)$.
Let $u_1 \in L_1$ and $u_2 \in L_2$.
Then the comultiplication formula implies
$$\Delta(\teone)(u_1 \otimes u_2) =
\teone u_1 \otimes q^{k_1 + k_2} u_2 \pm u_1 \otimes
\teone u_2  - (1-q^2) \tkone u_1 \otimes e_1 q^{2 k_1} u_2,$$
where $\pm$ is according that $u_1$ is even or odd.
It is obvious that the first two terms belong to $L_1 \otimes L_2$.
For the last term, we may assume that $u_2 = f_1^{(s)} v$ with
$e_1 v=0$. Then we have
\begin{equation*}
\begin{aligned}
e_1 q^{2k_1} u_2  & = e_1 q^{2k_1} f_1^{(s)} v = q^{2 \langle k_1,
\wt(v) - s \alpha_1 \rangle} [\langle k_1 - k_2, \wt(v) \rangle -
s +1] f_1^{(s-1)} v \\
& = q^{2 \langle k_1, \wt(v) \rangle -2s } [\langle k_1 - k_2,
\wt(v) \rangle -s + 1] \te_1 u_2 \\
&=\frac{q^{\langle 3k_1 - k_2, \wt(v) \rangle -3s +2} - q^{\langle
k_1 + k_2, \wt(v) \rangle -s}}{q^2 -1} \te_1 u_2.
\end{aligned}
\end{equation*}
Since
\eqn
&&\langle 3k_1 - k_2, \wt(v) \rangle -3s +2 = 3 (\langle k_1-k_2,
\wt(v) \rangle -s)+2 \langle  k_2, \wt(v) \rangle +2>0,\\
&&\langle k_1 + k_2 , \wt(v) \rangle -s =\lan k_1,\wt(u_2)\ran
+\lan k_2,\wt(v)\ran\ge\lan k_1,\wt(u_2)\ran\ge0,
\eneqn
If $\lan k_1,\wt(u_2)\ran=0$, then $f_1u_2=0$ and hence
$s=-\lan h_1,\wt(u_2)\ran= \lan k_2,\wt(u_2)\ran$. 
Thus we conclude 
\eq
&&\ba{l}
e_1 q^{2k_1} u_2\equiv \te_1u_2\pmod{L_2}\quad\text{if $\lan k_1,\wt(u_2)\ran=0$,}
\\[1ex]
e_1 q^{2k_1} u_2 \in qL_2\quad\text{if $\lan k_1,\wt(u_2)\ran>0$.}
\ea\label{eq:e1q}
\eneq
 Hence $L_1 \otimes L_2$ is
stable under $\Delta(\teone)$.

Similarly, one can show that $f_1 q^{k_1 + k_2 -1} L_2 \subset
L_2$, which implies $L_1 \otimes L_2$ is stable under
$\Delta(\tfone)$. Thus we have shown that $L_1\otimes L_2$
is stable under the Kashiwara operators. 

\medskip
We shall prove the tensor product rule. To prove the $\teone$-case,
let $u_1 \in l_{b_1}, u_2 \in l_{b_2}$,
and we consider the following three cases separately.

\vskip 3mm
\noindent
{\bf Case 1:} $\langle k_1, \wt(b_2) \rangle = \langle k_2,
\wt(b_2) \rangle =0$.

By the comultiplication formula, we have
$$\Delta(\teone)(u_1 \otimes u_2) = \teone u_1 \otimes u_2 \pm u_1
\otimes \teone u_2 - (1-q^2) \tkone u_1 \otimes e_1 u_2.$$
Since $\langle k_2, \wt(b_2) + \alpha_1 \rangle = \langle k_2, \wt(b_2)
+ \eps_1 - \eps_2 \rangle = -1 <0$, we must have $\teone u_2=e_1 u_2 =0$.
Hence $\Delta(\teone) (u_1 \otimes u_2) = \teone
u_1 \otimes u_2 $.

If $\teone =0$ on $l_{b_1}$, then $\teone \otimes 1 =0$ on $l_{b_1}
\otimes l_{b_2}$. If $\teone\cl l_{b_1} \rightarrow l_{\teone b_1}$ is
an isomorphism, then $\teone \otimes 1 \cl l_{b_1} \otimes l_{b_2}
\rightarrow l_{\teone b_1} \otimes l_{b_2}$ is also an isomorphism
as desired.

\medskip
\noindent
{\bf Case 2:} $\langle k_1, \wt(b_2) \rangle > 0$.

By the comultiplication formula and \eqref{eq:e1q}, we have
\begin{equation*}
\begin{aligned}
\Delta(\teone)(u_1 \otimes u_2) &
= \teone u_1 \otimes q^{\langle k_1 + k_2, \wt(b_2) \rangle} u_2
\pm u_1 \otimes\teone u_2 \\
& - (1-q^2) \tkone u_1 \otimes e_1q^{2k_1} u_2 \\
& \equiv \pm u_1 \otimes \teone u_2 \quad (\text{mod} \ q L_1 \otimes
L_2).\end{aligned}
\end{equation*}

\medskip
\noindent
{\bf Case 3:}  $\langle k_1, \wt(b_2) \rangle =0,  \ \ \langle k_2,
\wt(b_2) \rangle > 0$.

The comultiplication formula and \eqref{eq:e1q} yield
\begin{equation*}
\begin{aligned}
\Delta(\teone)(u_1 \otimes u_2) & = \teone u_1 \otimes q^{\langle
k_1 + k_2, \wt(b_2) \rangle} u_2 \pm u_1 \otimes \teone u_2 \\
& - (1-q^2) \tkone u_1 \otimes e_1q^{2k_1} u_2 \\
& \equiv \pm u_1 \otimes \teone u_2 -\tkone u_1 \otimes e_1 u_2 \quad
(\text{mod} \ q L_1 \otimes L_2).
\end{aligned}
\end{equation*} 
Since $\langle k_1, \wt(b_2) \rangle =0$ and $\tkone^2=(1-q^4)^{-1}(1-q^{4k_1})$, we have
$$k_{\bar 1}u_2 =0, \ \
\tkone^2 e_1 u_2 = \dfrac{1 - q^{4 k_1}}{1-q^4} e_1 u_2 = e_1 u_2.$$
It follows that
$$\teone u_2 = -q^{-1} (e_1 k_{\bar 1} - q k_{\bar 1} e_1) q^{k_1}
u_2 = k_{\bar 1} e_1 q^{k_1}u_2 = k_{\bar 1} q^{k_1 -1} e_1 u_2 = \tkone
e_1 u_2.$$ Hence we obtain
$$\tkone \teone u_2 = \tkone^2 e_1 u_2 = e_1 u_2, $$
which implies
$$
\begin{aligned}
\Delta(\teone)(u_1 \otimes u_2) &\equiv\pm u_1 \otimes \teone u_2 - \tkone
u_1 \otimes \tkone \teone u_2 \\
&\equiv 
(1 - \tkone \otimes \tkone) (1 \otimes \teone) (u_1 \otimes u_2).
\end{aligned}
$$
The operator $1 - \tkone \otimes \tkone$ on $l_{b_1}\otimes l_{\te_1b_2}$
is invertible because 
$(\tkone \otimes \tkone)^2=-\tkone^2 \otimes \tkone^2=
-(1-q^4)^{-1}(1 - q^{4 k_1})\otimes \id$ acts on $l_{b_1\otimes \teone b_2}$ by
the multiplication of a scalar different from $1$.
Hence the map $\Delta(\teone) \cl l_{b_1} \otimes l_{b_2}
\rightarrow l_{b_1} \otimes l_{\teone b_2}$, which is either 0 or an
isomorphism according that $\teone b_2=0$ or not.

The assertions on $\tfone$ can be verified in a similar manner.
The remaining property (b) (ii) in Definition~\ref{def:crystal base}
follows immediately from the formula \eqref{eq2:tensor product}.
\end{proof}

Motivated by the properties of crystal bases, we introduce the
notion of abstract crystals.

\Def An {\em abstract $\mathfrak{q}(n)$-crystal} is a
$\mathfrak{gl}(n)$-crystal together with the maps $\teone, \tfone\cl B
\to B \sqcup \{0\}$ satisfying the following conditions:
\bna
\item $\wt(B)\subset P^{\ge0}$,
\item $\wt(\teone b) = \wt(b) + \alpha_1$, $\wt(\tfone b) = \wt(b) -
\alpha_1$,

\item for all $b, b' \in B$, $\tfone b = b'$ if and only if $b = \teone b'$,

\item 
if $3\le i\le n-1$, we have
\be[{\rm(i)}]
\item
the operators $\teone$ and $\tfone$ commute with
$\te_i$ and $\tf_i$ ,

\item if $\teone b\in B$, then
$\eps_i(\teone b)=\eps_i(b)$ and $\vphi_i(\teone b)=\vphi_i(b)$.
\ee
\end{enumerate}
 \edf
Note that any crystal basis of $\uqqn$-modules in $\Oint$
satisfies the property (d) by Lemma~\ref{com:evenodd}.

Let $B_1$ and $B_2$ be abstract $\qn$-crystals. The {\em tensor
product} $B_1 \otimes B_2$ of $B_1$ and $B_2$ is defined to be the
$\mathfrak{gl}(n)$-crystal $B_1 \otimes B_2$ together with the maps
$\teone$, $\tfone$ defined by \eqref{eq2:tensor product}. Then it is
an abstract $\qn$-crystal.

The following associativity of the tensor product
is easily checked. 

\begin{prop}
Let $B_1, B_2$ and $B_3$ be abstract $\qn$-crystals. Then we have
$$(B_1 \otimes B_2) \otimes B_3 \simeq B_1\otimes (B_2 \otimes B_3).$$
\end{prop}

\begin{example} \hfill

\bna
\item If $(L, B, l_{B})$ is a crystal basis of a $\Uq$-module $M$ in the
category $\Oint$, then $B$ is an abstract
$\qn$-crystal.

\item The crystal graph $\B$ of the vector representation $\V$
is an abstract $\qn$-crystal.

\item By the tensor product rule, $\B^{\otimes N}$ is an abstract
$\qn$-crystal. When $n=3$, the $\qn$-crystal structure of $\B
\otimes \B$ is given below.

$$\xymatrix
{*+{\young(1) \otimes \young(1)} \ar[r]^1 \ar@{-->}[d]^{\ol 1} &
 *+{\young(2) \otimes \young(1)} \ar@<-0.5ex>[d]_1 \ar@{-->}@<0.5ex>[d]^{\ol 1} \ar[r]^2&
 *+{\young(3) \otimes \young(1)} \ar@<-0.5ex>[d]_1 \ar@{-->}@<0.5ex>[d]^{\ol 1} \\
 *+{\young(1) \otimes \young(2)} \ar[d]^2 &
 *+{\young(2) \otimes \young(2)} \ar[r]_2 &
 *+{\young(3) \otimes \young(2)} \ar[d]^2 \\
 *+{\young(1) \otimes \young(3)} \ar@{-->}@<-0.5ex>[r]_{\ol 1} \ar@<0.5ex>[r]^{1} &
 *+{\young(2) \otimes \young(3)} &
 *+{\young(3) \otimes \young(3)}
 }$$

\item For a strict partition $\la = (\la_1 > \la_2 > \cdots > \la_r
>0)$, 
let $Y_{\la}$ be the skew Young diagram having $\la_1$ many
boxes in the principal diagonal, $\la_2$ many boxes in the second
diagonal, etc.
For example, if $\la=(7 > 6 > 4 > 2 > 0)$, then we
have

$$Y_{\la} = \young(::::::\hfill,:::::\hfill\hfill,::::\hfill\hfill\hfill,:::\hfill\hfill\hfill\hfill,::\hfill\hfill\hfill\hfill,:\hfill\hfill\hfill,\hfill\hfill) \quad.$$

Let $\B(Y_{\la})$ be the set of all semistandard tableaux of shape
$Y_{\la}$ with entries from $1, 2, \ldots, n$. Then by an {\em
admissible reading} introduced in \cite{BKK}, $\B(Y_{\la})$ can be
embedded in $\B^{\otimes N}$, where $N=\la_1 + \cdots + \la_r$,
and it is stable under the Kashiwara operators $\tei,\tfi$ ($i=1,
\cdots, n-1,\ol1$). Hence it becomes an abstract
$\qn$-crystal. Moreover, the $\qn$-crystal structure thus obtained
does not depend on the choice of admissible reading.

Indeed, since $Y_\la$ is a skew Young diagram, it is stable under
the even Kashiwara operators, and the $\gl(n)$-crystal structure
does not depend on the choice of admissible reading. Let $T$ be a
semistandard tableau of shape $\la$ and let $\beta$ be the
lowest
box with entry $1$ in the principal diagonal of $T$. Since a
box with entry $1$ must lie in the principal diagonal of $T$, every
box with entry $1$ except $\beta$ lies in the northeast of
$\beta$. Let $\psi\cl \B(Y_\la) \rightarrow \B^{\tensor N}$ be an
admissible reading. It follows that $\beta$ is the rightmost box
with entry $1$ in $\psi(T)$. If there is a box, say $\gamma$, with
entry $2$ in the southwest of $\beta$ in $T$, then $\gamma$ must
appear after $\beta$ in $\psi(T)$. Thus we get $\tfone (\psi(T)) =0$. If
there is no box with entry $2$ in the southwest of $\beta$ in $T$,
then we know that every box with entry $2$ must lie in the
northeast of $\beta$ in $T$, and hence there is no box with entry
$2$ after $\beta$ in $\psi(T)$. Thus $\tfone$ acts on $\beta$.
Since the entry of the right box of $\beta$ in $T$ is greater than or
equal to $2$,
 we have $\tfone (\psi(T))=\psi(T')$, where $T'$ is the semistandard tableau of shape $\la$ obtained from $T$
 by replacing the entry of $\beta$ from $1$ to $2$.
It follows that $\B(Y_\la)$ is stable under the action of $\tfone$ and
it does not depend on the choice of admissible reading.

Let $\delta$ be the leftmost box with entry $2$ in $T$. If
$\delta$ lies in the second diagonal, the entry of the box lying
in the left of $\delta$ must be $1$. Then, for any admissible
reading $\psi$, $\teone \psi(T) = 0$. Thus we may assume that $\delta$ lies in
the principal diagonal of $T$,
and our assertion on $\teone$ follows from similar arguments as above.

In Figure~\ref{fi:B(Y_(3,1,0))}, we illustrate the crystal
$\B(Y_\la)$ for $n=3$ and $\la=(3>1>0)$. Note that it is
connected. However, in general, $\B(Y_\la)$ is not connected.
   \ee
\end{example}
\begin{figure}[!h]
 $$\scalebox{.8}{\xymatrix@R=1pc@H=1pc{ & &  {\young(::1,:12,1)} \ar[dl]_1 \ar_2[d]  \ar^{\ol 1}@{-->}[dr] & & & \\
& {\young(::1,:22,1)} \ar@<-0.5ex>_1[dl] \ar@<1ex>^{\ol 1}@{-->}[dl]
\ar_2[d] & {\young(::1,:13,1)} \ar_1[d]\ar^{\ol 1}@{-->}[dr] &
{\young(::1,:12,2)} \ar_2[d] & & \\
{\young(::1,:22,2)} \ar_2[d] &
{\young(::1,:23,1)}\ar@<-0.5ex>_1 [dl] \ar@<1ex>^{\ol 1}@{-->}[dl] \ar_2[d]
& {\young(::2,:13,1)}  \ar_1 [d] \ar^{\ol 1}@{-->}[dr] & {\young(::1,:13,2)}
\ar_1 [d] \ar_2[dr] & & {\young(::1,:12,3)} \ar@<-0.5ex>_1 [d]
\ar@<1ex>^{\ol 1}@{-->}[d] \\
 {\young(::1,:23,2)} \ar_2[d] & {\young(::1,:33,1)} \ar^{\ol 1}@{-->}[dl] \ar_1 [d] & {\young(::2,:23,1)} \ar@<-0.5ex>_1 [d]
\ar@<1ex>^{\ol 1}@{-->}[d] \ar_2[dl] & {\young(::2,:13,2)}
 \ar_2[d] & {\young(::1,:13,3)} \ar_1 [dl] \ar^{\ol 1}@{-->}[dr] & {\young(::1,:22,3)} \ar_2[d] \\
 {\young(::1,:33,2)} \ar_2[dr] & {\young(::2,:33,1)} \ar@<-0.5ex>@{-->}_{\ol 1}[dr] \ar@<1ex>^{1}[dr]  &{\young(::2,:23,2)} \ar_2[d] & {\young(::2,:13,3)} \ar@<-0.5ex>_1 [d] \ar@<1ex>^{\ol 1}@{-->}[d] & & {\young(::1,:23,3)} \\
   & {\young(::1,:33,3)} \ar@<-0.5ex>@{-->}_{\ol 1}[dr] \ar@<1ex>^{1}[dr]       &  {\young(::2,:33,2)} \ar_2[d] &   {\young(::2,:23,3)} &  &  \\
   & &{ \young(::2,:33,3)}&& &}}$$
\caption{${\mathbf B}(Y_\la)$ for $n=3$, $\la = (3>1>0)$} \label{fi:B(Y_(3,1,0))}
   \end{figure}

Let $B$ be an abstract $\qn$-crystal. For $i = 1,2,\ldots, n-1$,
we define the automorphism $S_i$ on $B$ by
\eq
&&S_i b =
\begin{cases}
\tfi^{\langle h_i, \wt b \rangle} b & \text{if} \quad {\langle h_i, \wt b \rangle} \geq 0, \\
\tei^{-\langle h_i, \wt b \rangle} b & \text{if} \quad {\langle h_i, \wt b \rangle} \leq 0.
\end{cases}\label{def:Sicr}
\eneq
Let $w$ be an element of the Weyl group $W$ of $\gl(n)$. Then, as
shown in \cite{Kas94}, there exists a unique action $S_{w} \cl B \to
B$ of $W$ on $B$ such that $S_{s_i}=S_i$ for $i=1,2,\ldots,n-1$. Note that
$\wt(S_w b)=w(\wt(b))$ for any $w \in W$ and $b \in B$.

For $i=1, \ldots, n-1$, we set
\eq
&&w_i = s_2 \cdots s_{i} s_1 \cdots s_{i-1}.
\label{def:wi}
\eneq
Then $w_i$ is the shortest element in $W$ such that $w_i(\alpha_i) = \alpha_1$.
We define the {\em odd Kashiwara operators} $\teibar$, $\tfibar$ $(i=2, \ldots, n-1)$ by
$$\teibar = S_{w_i^{-1}} \teone
S_{w_i}, \ \ \tfibar = S_{w_i^{-1}} \tfone S_{w_i}.$$
We say that
$b \in B$ is a {\em highest weight vector} if $\tei b = \teibar b
=0$ for all $i=1, \ldots, n-1$.

\vs{2ex}
\begin{remark}\label{rem:teibar}
These actions can be lifted to actions on $\uqqn$-modules.
Let $M$ be a $\uqqn$-module in $\Oint$.
For each $i=1,\ldots, n-1$, we have
$$M=\soplus_{\substack{\ell\ge k\ge0,\\\la\in P,\; \lan h_i,\la\ran=\ell}}
f_i^{(k)}\bl(\Ker(e_i)_\la\br).
$$
Hence we can define the endomorphism $S_i$ of $M$ by
\eq
S_i(f_i^{(k)}u)=f_i^{(\ell-k)}u\quad\text{for $u\in\Ker(e_i)_\la$.}
\label{def:Si}
\eneq
Then $S_i^2=\id_M$ and we have
$S_i(M_\la)=M_{s_i\la}$.
If $(L,B,l_B)$ is a crystal basis of $M$, then
$L$ is stable under $S_i$, and
$S_i$ induces an action on $L$ and $L/qL$.
Obviously, we have $S_i(l_b)=l_{S_ib}$ for $b\in B$, where $S_ib$
is defined in \eqref{def:Sicr}.
We define the endomorphisms $\te_{\ol i}$ and $\tf_{\ol i}$ of $M$ by
\eq
\ba{rcl}
\te_{\ol i}&=&(S_2\cdots S_iS_1\cdots S_{i-1})^{-1}\circ
\teone\circ(S_2\cdots S_iS_1\cdots S_{i-1}),\\[1ex]
\tf_{\ol i}&=&(S_2\cdots S_iS_1\cdots S_{i-1})^{-1}\circ
\tfone\circ(S_2\cdots S_iS_1\cdots S_{i-1}).
\ea\label{def:efibar}
\eneq
Then we have
$$\text{$\te_{\ol i}M_\mu\subset M_{\mu+\alpha_i}$ and
$\tf_{\ol i}M_\mu\subset M_{\mu-\alpha_i}$ for every $\mu\in P^{\ge0}$.}$$
Let $(L,B,l_B)$ be a crystal basis of $M$.
Then $L$ is stable under the action of $\te_{\ol i}$, and
$\te_{\ol i}$ induces an action on $L/qL$,
and we have
\eqn
&&\left\{
\parbox{\mylength}{
(i) if $\te_{\ol i}b\not=0$, then $\te_{\ol i}$ induces an isomorphism
$l_{b}\isoto l_{\te_{\ol i}b}$,\\[1.5ex]
(ii) if $\te_{\ol i}b=0$, then $\te_{\ol i}(l_b)=0$.
}
\right.
\eneqn
Similar properties hold for $\tf_{\ol i}$. Note that
$$
\begin{aligned}
\Ker(\teibar\cl L/qL\to L/qL) =& \Ker (\teone S_{w_i}) = S_{w_i}^{-1}(\Ker \teone)
= S_{w_i^{-1}}(\Ker \teone)\\
=& S_{w_i^{-1}} \Big( \soplus_{\teone b=0} l_{b} \Big)
             = \soplus_{\teone b =0} l_{S_{w_i^{-1}} b}
=\soplus_{\teone S_{w_i} b =0} l_b = \soplus_{\teibar b = 0} l_b .
\end{aligned}
$$
\end{remark}

\begin{example}
Let $\la$ be a strict partition.
Observe that $\B(Y_\la)$ has a unique element of weight $\la$, say $b_{Y_\la}$.
Since $\la + \alpha_i \notin \wt(\B(Y_\la))$ for any $i=1, 2, \ldots, n-1$,
$b_{Y_\la}$ is a highest weight vector.
Thus, for each admissible reading $\psi$, we see that $\psi(b_{Y_\la})$ is a highest weight vector
in $\B^{\tensor N}$.
\end{example}

\begin{lemma} \label{le:existence of h.w. vectors}
Every abstract $\q(n)$-crystal contains a highest weight vector.
\begin{proof}
Recall that $\la \in \wt(B)\seteq\set{\wt(b)}{b \in B}$
is called {\it maximal} if
$\la + \alpha_i \notin \wt(B)$ for $i=1,2,\ldots,n-1$.
Since $\wt(\teibar b)=\wt(b)+ \alpha_i $,
a vector in a crystal $B$ with a maximal weight is a highest weight vector.
Because $\wt(B)$ is a finite set,
there exists a maximal element $\lambda$ so that
we have an element $b \in B$ with a maximal weight $\lambda$.
\end{proof}
\end{lemma}

\begin{remark} \label{re:Clifford algebras}
\bna
\item Let $\la$ be a strict partition with $\ell(\la) =r$
and let $M$ be a highest weight module of
highest weight $\la$ in $\Oint$. Set $\tki =
q^{k_i-1} k_{\ol i}$ for $i=1,2,\ldots, n$. 
Since $M \in \Oint$, we have $\tki=0$ on $M_\la$ for $i  > r$.
Note that $\tki^2=\dfrac{1-q^{4\la_i}}{1-q^4}$ on $M_\la$
and
 $\Big (\dfrac{1-q^{4\la_i}}{1-q^4} \Big)^{-\frac{1}{2}} \in \A \subset \F$.
Let
$$C_i \seteq \Big (\dfrac{1-q^{4\la_i}}{1-q^4} \Big)^{-\frac{1}{2}} \tki.$$
Then on $M_\la$, we have
\eq \label{defining relations of Clifford algebra}
C_i^2=1,  \ C_i C_j + C_j C_i = 0 \ (i \neq j).
\eneq
Thus $M_\lambda$ can be regarded as a module over $\F[C_1,\ldots,C_r]$,
where $\F[C_1,\ldots,C_r]$ is the associative $\F$-algebra generated by $\set{C_i}{i=1,2,\ldots,r}$
with the defining relations \eqref{defining relations of Clifford algebra}.

\item
Let $\C[C_1,\ldots,C_r]$ and $\A[C_1,\ldots,C_r]$ be the
associative $\C$-algebra  and $\A$-algebra, respectively,
generated by $\set{C_i}{i=1,2,\ldots,r}$ with the defining
relations \eqref{defining relations of Clifford algebra}. For a
superring
 $R$, we define $\Mod(R)$ and $\SMod(R)$ to be the category of
$R$-modules and the category of $R$-supermodules, respectively. 

If $r$ is odd, then we have the following commutative diagram :
$$
\xymatrix{
\Mod(\A) \ar[r]^(.34){\sim} \ar[d]^{\F \tensor_\A ( - )} & \SMod(\A[C_1,\ldots, C_r]) \ar[d]^{\F \tensor_\A ( - )}\\
\Mod(\F) \ar[r]^(.34){\sim}        & \SMod(\F[C_1,\ldots, C_r]),
}
$$
If $r$ is even, then we have the following commutative diagram :
$$\xymatrix{
\SMod(\A) \ar[r]^(.36){\sim} \ar[d]^{\F \tensor_\A ( - )} & \SMod(\A[C_1,\ldots, C_r]) \ar[d]^{\F \tensor_\A ( - )}\\
\SMod(\F) \ar[r]^(.36){\sim}        & \SMod(\F[C_1,\ldots, C_r]).
}
$$
In both cases, the horizontal arrows are given by
$$K \mapsto V \tensor_\C K$$
for each module $K$ in the left hand side,
where  $V$ denotes an irreducible supermodule over $\C[C_1,\ldots,C_r]$.
\ee
\end{remark}

To summarize, we obtain the following proposition.
\begin{prop} \label{prop:uniqueness of lattice of highest weight space of an irreducible module}
\bna
\item For a strict partition $\la \in \Lambda^+$ with $l(\la)=r$, let $\rm HT(\la)$ be the category of highest weight modules with highest weight $\la$
in $\Oint$.
Then $\rm HT(\la)$ is equivalent
to $\SMod(\F[C_1,\ldots,C_r])$, where the equivalence is given by
$${\rm HT(\la)} \ni M \mapsto M_\la \in \SMod(\F [C_1,\ldots,C_r]).$$
In particular, the homomorphism
$\End_{\uqqn}(M)\to\End_{\F[C_1,\ldots,C_r]}(M_\la)$
is an isomorphism  for any $M\in \rm HT(\la)$.
\item  For a $\uqqn$-module $M \in \rm HT(\la)$, let $L$, $L'$
be finitely generated free $\A$-submodules of $M_\la$
such that

\quad {\rm (i)} $L$ and $L'$ are stable under $\tki$'s
$(i=1,2,\ldots, n)$,

\quad {\rm (ii)} $\F \otimes_\A L \simeq \F \otimes_\A L'\simeq M_\la$.
 \ee
Then there exists a $\uqqn$-module automorphism $\vphi$ of $M$ such
that $\vphi L = L'$.
\end{prop}

\vskip 1cm

\section{Highest weight vectors in $\B^{\tensor N}$}

In this Section, we will give algebraic and combinatorial
characterizations of highest weight vectors in the abstract $\q(n)$-crystal
$\B^{\tensor N}$.

\begin{definition}
Let $B$ be an abstract $\q(n)$-crystal.
\bni
\item An element $b \in B$ is called a \emph{$\gl(a)$-highest weight vector} if $\te_i b = 0$ for $1 \leq i < a \leq n$.
\item An element $b \in B$ is called a \emph{$\q (a)$-highest weight vector} if $\te_i b = \teibar b = 0$ for $1 \leq i < a \leq n$.
\end{enumerate}
\end{definition}
\noindent
In particular, a highest weight vector in $B$ is a $\q(n)$-highest weight vector.\\

From now on, we denote by $\tensor_{j \geq m \geq i} (r_1 \ r_2
\cdots r_m )^{\tensor y_m} $ the following vector in $\B^{\tensor
N}$ :
$$\begin{array}{l}
\underbrace{(r_1 \tensor \cdots \tensor r_j) \tensor
\cdots \tensor (r_1 \tensor \cdots \tensor r_j)}_{y_j - \rm{times}}
\tensor
\underbrace{(r_1 \tensor \cdots \tensor r_{j-1}) \tensor
\cdots \tensor (r_1 \tensor \cdots \tensor r_{j-1})}_{y_{j-1} - \rm{times}}
\tensor \\[.5ex]
\cdots \tensor \underbrace{(r_1 \tensor \cdots \tensor r_{i+1}) \tensor
\cdots \tensor (r_1 \tensor \cdots \tensor r_{i+1})}_{y_{i+1} - \rm{times}}
\tensor \underbrace{(r_1 \tensor \cdots \tensor r_i) \tensor
\cdots \tensor (r_1 \tensor \cdots \tensor r_i)}_{y_i - \rm{times}},
\end{array}
$$
where $N = \sum^{j}_{m=i} m y_m$.

Let $b$ be an element of a $\gl(n)$-crystal $B$. We denote
by $C(b)$ the connected component of $B$ containing $b$.

\begin{definition}
Let $B_i$ be a $\gl(n)$-crystal and let $b_i \in B_i$
$(i=1,2)$. We say that $b_1$ is \emph{$\gl(n)$-crystal equivalent}
to $b_2$ if there exists an isomorphism of $\gl(n)$-crystals $C(b_1)
\isoto C(b_2)$ sending $b_1$ to $b_2$.
\end{definition}

Recall that $w_i = s_2 \cdots s_{i} s_1 \cdots s_{i-1}$.
\begin{lemma} \label{le:swib0}
Let $B$ be a $\gl(n)$-crystal.
 \bna
\item A vector $b_0$ in $\B \tensor B$ is a
$\gl(n)$-highest weight vector if and only if
 $b_0=1 \tensor \tf_1 \cdots \tf_{j-1} b$ for some $j \in \{1,2,\ldots, n\}$ and
some $\gl(n)$-highest weight vector $b \in B$ such that $\wt(b_0) =
\wt(b)+\eps_j$ is a partition.

\item Let $b$ be a $\gl(n)$-highest weight vector in $B$ and 
$j \in \{1,2,\ldots, n\}$.
If $\wt(b)+\eps_j$ is a partition, then
$b_0=1\tensor \tf_1 \cdots \tf_{j-1} b$ is a $\gl(n)$-highest weight
vector in $\B \tensor B$ and we have
$$
S_{w_i} b_0 =
\begin{cases}
3 \tensor \tf_3 \cdots \tf_{j+1} S_{w_i} b & \text{if} \quad j+1 \leq i < n, \\
1 \tensor S_{w_i} b & \text{if} \quad i=j, \\
1 \tensor \tf_1 S_{w_i} b & \text{if} \quad i=j-1,
\end{cases}
$$
 and
$$
S_{u_i}b_0 = 1 \tensor \tf_1 \tf_2 S_{u_i} b' \quad \text{if} \quad i \leq j-2,
$$
where $z_i = s_3 s_4 \cdots s_{i+1}$, $u_i = z_i w_i $ and $b' = \tf_{i+2} \cdots \tf_{j-1} b$.
\ee
\begin{proof}
{\rm (a)}
For a partition $\la$, let us denote by $B_{\gl(n)}(\la)$ the crystal graph of the highest weight $\gl(n)$-module
with highest weight $\la$.
It is enough to show that the assertion holds for $B=B_{\gl(n)}(\la)$ for any partition $\la$.

Let $b_0=1 \tensor \tf_1 \cdots \tf_{j-1} b$ for some
$\gl(n)$-highest weight
 vector $b \in B$ such that $\wt(b_0)$ is a partition.
Since any two $\gl(n)$-highest weight vectors with the same
highest weight are $\gl(n)$-crystal equivalent,
by embedding $B$ to $\B^{\tensor N}$ for some $N$,
we may assume that
$b=\tensor_{n \geq m \geq 1} (12 \cdots m)^{\tensor x_m}$, where
$x_m=\langle k_m-k_{m+1} , \wt(b)  \rangle$ for $1 \leq m \le n-1$.
Since $\wt(b)+\eps_j = \wt(b_0)$ is a partition, we have
$x_{j-1} \geq 1$. Thus we have
\begin{equation} \label{eq:gln h.w. vectors}
\begin{array}{ll}
&1 \tensor \tf_1 \tf_2 \cdots \tf_{j-1} b = \\
& \qquad1 \tensor \tensor_{m \geq j} (1 \cdots m)^{\tensor x_m}
\tensor (2 3 \cdots j) \tensor \tensor_{j-1 \geq m \geq 1}(1 \cdots
m)^{\tensor(x_m-\delta_{m,{j-1}})},
\end{array}
\end{equation}
which is a $\gl(n)$-highest weight vector in $\B \tensor B$.
Since we have
$$\B \tensor B \simeq \soplus_{ \la + \eps_j  : \, \, partiton} B_{\gl(n)}(\la+\eps_j), $$
the number of highest weight vectors in $\B \tensor B$ is the same as
the number of vectors of the form in \eqref{eq:gln h.w. vectors}.

\bigskip
\noi
{\rm(b)} 
We may assume that $b=\tensor_{n \geq m \geq 1} (12 \cdots m)^{\tensor x_m}$
as above.
Then by \eqref{eq:gln h.w. vectors}, we have 
\begin{equation}\label{eq:b0}
b_0 = 1 \tensor \tensor_{m \geq j} (1 2\cdots m)^{\tensor x_m} \tensor (2 3 \cdots j) \tensor
                           \tensor_{j-1 \geq m \geq 1}(1 2 \cdots m)^{\tensor (x_m-\delta_{m,j-1})}.
\end{equation}
We also have
\begin{equation} \label{eq:swib}
\begin{aligned}
S_{w_i}b = & \tensor_{m \geq i+1} (1 2\cdots m)^{\tensor x_m} \tensor(1 3 4 \cdots i+1)^{\tensor x_i} \tensor \\
          & \tensor_{i-1 \geq m \geq 1}(3 4 \cdots m+2)^{\tensor x_m}.
\end{aligned}
\end{equation}
Here we used the following facts :
\begin{enumerate}
\item For $w \in W$ and $\gl(n)$-highest weight vectors $b_1$ and $b_2$,
$$S_w(b_1 \tensor b_2) = S_w b_1 \tensor S_w b_2. $$
\item Suppose that $0 < a_1 < a_2 < \cdots < a_r \leq n $, $0 < x_1 < x_2 < \cdots < x_r \leq n$ and
$w(\{a_1, \ldots a_r \})=\{x_1, \ldots x_r\}$.
Then we have
$$S_w (a_1 \otimes \cdots \otimes a_r) = x_1 \otimes \cdots \otimes x_r.$$
\end{enumerate}

\noi
{\bf Case 1:} $j+1 \leq i < n$. \\
From \eqref{eq:b0}, we have
\begin{equation*}
\begin{array}{lll}
S_{w_i}b_0 =&3 \tensor \tensor_{m \geq i+1} (1 2\cdots m)^{\tensor x_m} \tensor(1 3 4 \cdots i+1)^{\tensor x_i} \tensor  \\
             & \tensor_{i-1 \geq m \geq j}(3 4 \cdots m+2)^{\tensor x_m} \tensor (4 5 \cdots j+2) \tensor \\
             & \tensor_{j-1 \geq m \geq 1} (3 \cdots m+2)^{\tensor(x_m-\delta_{m, j-1})}.
\end{array}
\end{equation*}
On the other hand, from \eqref{eq:swib}, we have
\begin{equation*}
\begin{array}{lll}
\tf_3 \cdots \tf_{j+1} S_{w_i}b =&\tensor_{m \geq i+1} (1 2\cdots m)^{\tensor x_m} \tensor(1 3 4 \cdots i+1)^{\tensor x_i} \tensor \\
                                 &\tensor_{i-1 \geq m \geq j}(3 4 \cdots m+2)^{\tensor x_m} \tensor (4 5 \cdots j+2) \tensor \\
                                 &\tensor_{j-1 \geq m \geq 1}(3 4 \cdots m+2)^{\tensor(x_m-\delta_{m,j-1})}.
\end{array}
\end{equation*}
Thus we get $$S_{w_i} b_0 = 3 \tensor \tf_3 \cdots \tf_{j+1} S_{w_i} b.$$

\noi
{\bf Case 2:} $i=j$. \\
From \eqref{eq:b0} and \eqref{eq:swib}, we have
\begin{equation*}
\begin{array}{lll}
 S_{w_i} b_0 &= S_{w_i}\big(1\tensor \tensor_{m \geq j}(1\cdots m)^{\tensor x_m} \tensor(2 \cdots j)
               \tensor\tensor_{j-1 \geq m \geq 1}(1 \cdots m)^{\tensor (x_m-\delta_{m,j-1})}\big)\\
             &= S_2 \cdots S_j \big(1\tensor \tensor_{m \geq j}(1\cdots m)^{\tensor x_m} \tensor
               \tensor_{j-1 \geq m \geq 1}(2 \cdots m+1)^{\tensor x_m}\big) \\
             &= 1\tensor \tensor_{m \geq j+1}(1\cdots m)^{\tensor x_m} \tensor(1 3 4 \cdots j+1)^{\tensor x_j} \tensor
               \tensor_{j-1 \geq m \geq 1}(3 \cdots m+2)^{\tensor x_m} \\
             &=1 \tensor S_{w_i} b.
\end{array}
\end{equation*}

\noi
{\bf Case 3:} $i=j-1$. \\
From \eqref{eq:b0}, we have
\begin{equation*}
\begin{array}{lll}
 S_{w_i} b_0 &=& S_2 \cdots S_{j-1} \big(1\tensor \tensor_{m \geq j}(1\cdots m)^{\tensor x_m} \tensor
               (2 \cdots j) \tensor (1 2 \cdots j-1)^{\tensor (x_{j-1} -1)} \tensor \\
             & & \tensor_{j-2 \geq m \geq 1}(2 \cdots m+1)^{\tensor x_m}\big) \\
             &=& 1\tensor \tensor_{m \geq j}(1\cdots m)^{\tensor x_m} \tensor (2 \cdots j) \tensor
             \tensor(1 3 4 \cdots j)^{\tensor (x_{j-1}-1)} \tensor \\
             &&\tensor_{j-2 \geq m \geq 1}(3 \cdots m+2)^{\tensor x_m}.

\end{array}
\end{equation*}
On the other hand, from \eqref{eq:swib}, we have
\begin{equation*}
\begin{array}{lll}
 &1 \tensor \tf_1 S_{w_i} b \\
 &= 1 \tensor \tf_1 \big( \tensor_{m \geq j}(1\cdots m)^{\tensor x_m} \tensor
               (1 3 4  \cdots j)^{\tensor x_{j-1}} \tensor \tensor_{j-2 \geq m \geq 1}(3 \cdots m+2)^{\tensor x_m}\big) \\
 &= 1\tensor \tensor_{m \geq j}(1\cdots m)^{\tensor x_m} \tensor (2  \cdots j)\otimes
               (1 3 4  \cdots j)^{\tensor (x_{j-1}-1)} \tensor \tensor_{j-2 \geq m \geq 1}(3 \cdots m+2)^{\tensor x_m}.
\end{array}
\end{equation*}
Hence we get
\begin{equation*}
S_{w_i} b_0 = 1 \tensor \tf_1 S_{w_i} b.
\end{equation*}

\noi
{\bf Case 4:} $i \leq j-2$. \\
Note that
\begin{equation*}
u_i(m) =
\begin{cases}
 m+3 & 1 \leq m < i, \\
 1 & m= i, \\
 2 & m=i+1 , \\
 3 & m=i+2 , \\
 m & m \geq i+3.
\end{cases}
\end{equation*}

We have
\begin{equation*}
\begin{array}{l}
S_{u_i} b_0 \\
= S_3 \cdots S_{i+1}
\big(1 \tensor \tensor_{m \geq j}(1 \cdots m)^{\tensor x_m}
     \tensor (2 \cdots j) \tensor
     \tensor_{j-1 \geq m \geq i+1} (1 \cdots m)^{\tensor (x_m -\delta_{m, j-1})} \\[1ex]
\hs{15ex} \tensor (1 3 \cdots i+1)^{\tensor x_i} \tensor
     \tensor_{i-1 \geq m \geq 1} (3 \cdots m+2)^{\tensor x_m} \big) \\[1.5ex]
=1 \tensor \tensor_{m \geq j}(1 \cdots m)^{\tensor x_m} \tensor (2 \cdots j)
      \tensor_{j-1 \geq m \geq i+2} (1 \cdots m)^{\tensor (x_m -\delta_{m, j-1})} \\[1ex]
 \hs{10ex} \tensor (1 2 4 \cdots i+2)^{\tensor x_{i+1}} \tensor (1 4 \cdots i+2)^{\tensor x_i}
    \tensor \tensor_{i-1 \geq m \geq 1} (4 \cdots m+3)^{\tensor x_m}.
\end{array}
\end{equation*}

On the other hand, we have
\begin{equation*}
\begin{array}{l}
\tf_1 \tf_2 ( S_{u_i} b') \\[1ex]
\hs{3ex}= \tf_1 \tf_2  S_{u_i} \big(\tensor_{m \geq j}(1 \cdots m)^{\tensor x_m}
   \tensor (1 \cdots i+1 \ i+3 \cdots j) \\
\hs{25ex} \tensor
    \tensor_{j-1 \geq m \geq 1} (1 \cdots m)^{\tensor (x_m -\delta_{m, j-1})} \big) \\[1ex]
\hs{3ex}=\tf_1 \tf_2 \big( \tensor_{m \geq j}(1 \cdots m)^{\tensor x_m}
         \tensor (1 2 4 \cdots j) \tensor
         \tensor_{j-1 \geq m \geq i+2} (1 \cdots m)^{\tensor (x_m -\delta_{m, j-1})} \\
\hs{15ex}  \tensor (1 2 4 \cdots i+2)^{\tensor x_{i+1}}
   \tensor (1 4 \cdots i+2)^{\tensor x_i} \tensor
   \tensor_{i-1 \geq m \geq 1} (4 \cdots m+3)^{\tensor x_m}  \big) \\[1ex]
\hs{3ex}=\tensor_{m \geq j}(1 \cdots m)^{\tensor x_m}
         \tensor (2 3 4 \cdots j) \tensor
         \tensor_{j-1 \geq m \geq i+2} (1 \cdots m)^{\tensor (x_m -\delta_{m, j-1})} \\
\hs{15ex}  \tensor (1 2 4 \cdots i+2)^{\tensor x_{i+1}}
   \tensor (1 4 \cdots i+2)^{\tensor x_i} \tensor
   \tensor_{i-1 \geq m \geq 1} (4 \cdots m+3)^{\tensor x_m}.
\end{array}
\end{equation*}

Thus, we obtain
\begin{equation*}
S_{u_i} b_0 = 1 \tensor \tf_1 \tf_2 S_{u_i} b'.
\end{equation*}
\end{proof}
\end{lemma}

\begin{lemma}\label{lem:e1f1}
Assume that $b\in \B^{\tensor N}$ satisfies
$\tf_1b\not=0$ and $\teone \tf_1b=0$.
Then $\teone b=0$.
\end{lemma}
\Proof
If $b$ does not contain  $2$, then it is trivial.
Assume that $b$ contains $2$ and $\teone b\not=0$. Then we can write
$b=b_1\otimes 2\otimes b_2$ such that $b_2$ contains neither $1$ nor $2$.
Since $\tf_1b\not=0$, we have
$\tf_1b=(\tf_1b_1)\otimes 2\otimes b_2$ and $\tf_1b_1\not=0$.
Therefore, $\teone \tf_1b=(\tf_1b_1)\otimes 1\otimes b_2$ does not vanish,
which is a contradiction.
\QED
\begin{theorem} \label{th:h.w. vectors}
Suppose that $b$ is a $\gl(n)$-highest weight vector in $\B^{\tensor (N-1)}$
and $b_0=1 \tensor \tf_1 \cdots \tf_{j-1} b$ is a highest
weight vector in $\B^{\tensor N}$. Then $b$ is
a highest weight vector in $\B^{\tensor(N-1)}$.
\end{theorem}
\begin{proof}
We shall prove $\teibar b =0 $ for $1 \leq i < n$.
%
%

\bigskip
\noindent
{\bf Case 1:} $j+1 \leq i <n$.

By Lemma~\ref{le:swib0}, we have $$S_{w_i} b_0 = 3 \tensor \tf_3 \cdots \tf_{j+1} S_{w_i} b.$$
Since $0=\te_{\ol 1} S_{w_i} b_0= \te_{\ol 1} (3 \tensor \tf_3 \cdots \tf_{j+1} S_{w_i} b)$,
we obtain $\te_{\ol 1} S_{w_i} b =0$. \\

\noindent
{\bf Case 2:} $i=j$.

We have
\begin{equation*}
 S_{w_i} b_0 =1 \tensor S_{w_i} b.
\end{equation*}
Since $0=\te_{\ol 1} S_{w_i} b_0= \te_{\ol 1} (1 \tensor S_{w_i} b)$,
we get $\te_{\ol 1} S_{w_i} b =0$.

\bigskip
\noindent
{\bf Case 3:} $i=j-1$.

Since
\begin{equation*}
S_{w_i} b_0 = 1 \tensor \tf_1 S_{w_i} b,
\end{equation*}
we have
\begin{equation*}
\te_{\ol 1} \tf_1 S_{w_i} b = 0.
\end{equation*}
Hence Lemma~\ref{lem:e1f1} implies $\te_{\ol 1} S_{w_i} b = 0$.

\bigskip
\noindent
{\bf Case 4:} $i \leq j-2$.

Set $b' \seteq \tf_{i+2} \cdots \tf_{j-1} b$.
Then $\te_k b' =0$ for $k \leq i+1$.
Hence $b'$ is a $\gl(i+2)$-highest weight vector.
Since $u_i^{-1}(\alpha_1)$ and $u_i^{-1}(\alpha_2)$ are positive roots,
$S_{u_i} b'$ is a $\gl(3)$-highest weight vector.
Here we have used the fact: 
\eq&&
\parbox{\mylength}{if $b$ is a $\gl(n)$-highest weight vector
and $w^{-1}(\alpha_i)$ is a positive root for $w\in W$
and $i$, then $\te_iS_wb=0$.}\label{eq:wh}
\eneq
For the same reason, $S_{u_i} b_0$ is a $\gl(n)$-highest weight vector.

By Lemma~\ref{le:swib0}, we have
\begin{equation*}
S_{u_i}b_0 = 1 \tensor \tf_1 \tf_2 S_{u_i} b'.
\end{equation*} 
Since $\te_{\ol 1}$ commutes with $S_3, \ldots S_{n-1}$,
$\te_{\ol1}$ commutes with $S_{z_i}$.
Hence
\begin{equation*}
\te_{\ol 1} S_{u_i} b_0 = \te_{\ol 1} S_{z_i} S_{w_i} b_0 = S_{z_i} \te_{\ol 1} S_{w_i} b_0 = 0.
\end{equation*}
Since $w_2 u_i = z_i w_{i+1}$, we also have
\begin{equation*}
\te_{\ol 1} S_{w_2} S_{u_i} b_0 = \te_{\ol 1} S_{z_i} S_{w_{i+1}}
b_0 =S_{z_i} \te_{\ol 1} S_{w_{i+1}} b_0 =0.
\end{equation*}
Thus $S_{u_i} b_0$ is a $\q(3)$-highest weight vector.
By Lemma~\ref{le:q3hw} below, we have $\te_{\ol 1} S_{u_i} b' =0.$
Since $\te_{\ol 1}$ commutes with $S_{z_i}$, we get
$\te_{\ol1} S_{z_i} S_{w_i} b' = S_{z_i} \te_{\ol 1} S_{w_i} b'$, and hence we conclude
$\te_{\ol i} b'= 0$.

On the other hand,
$\te_{\ol i}$ commutes with $\te_{j-1}\cdots\te_{i+2}$,
because $\te_k$ ($k\ge i+2$) commutes with
$S_1,\ldots, S_i$ and $\teone$.
Hence $\te_{j-1}\cdots\te_{i+2}$ commutes with $\te_{\ol i}$.
Since $b=\te_{j-1}\cdots\te_{i+2}b'$,
we obtain $\te_{\ol i}b=\te_{\ol i}\te_{j-1}\cdots\te_{i+2}b'=
\te_{j-1}\cdots\te_{i+2}\te_{\ol i}b'=0$. 
\end{proof}

\begin{lemma} \label{le:q3hw}
Suppose that $b$ is a $\gl(3)$-highest
weight vector in $\B^{\tensor (N-1)}$
and $b_0=1 \tensor \tf_1 \tf_2 b$ is a $\q(3)$-highest
weight vector in $\B^{\tensor N}$. Then $\te_{\ol 1} b =0$.
\begin{proof}
If $\te_{\ol1} b \neq 0$, then $b=b_1 \tensor 2 \tensor b_2$, where
$b_2$ contains neither $1$ nor $2$.
Since $\te_{\ol 1} b_0 =0$, we have $\te_{\ol 1} \tf_1 \tf_2 b=0$ and
hence Lemma~\ref{lem:e1f1} implies $\te_{\ol 1} \tf_2 b=0$.
It follows that $\tf_2 b = b_1\tensor 3 \tensor b_2$.
Hence $\te_{\ol 1}\tf_2 b=0$ implies $$\te_{\ol 1}b_1=0.$$
Moreover, $\tf_2(b_1 \tensor 2 \tensor b_2)=b_1\tensor 3 \tensor b_2$
implies that $\varphi_2(b_1)=0$ and $b_2$ does not contain $3$.
Since $\varepsilon_2(b_1)=0$, we conclude that
$b_1$ is $\gl(3)$-crystal equivalent to $1^{\tensor x}$ for some
positive integer $x$. Thus we get
\begin{equation}
\begin{array}{rl}
S_2 S_1 b_0 &= S_2 S_1 (1 \tensor \tf_1 \tf_2 b)
            = S_2 S_1 (1 \tensor \tf_1 b_1 \tensor 3\otimes b_2) \\
            &= S_2 (1 \tensor S_1 b_1 \tensor 3\otimes b_2)
            = 1 \tensor \te_2 S_2 S_1 b_1 \tensor 3\otimes b_2.
\end{array}\label{eq:s2s1}
\end{equation}
Here the third equality follows from
$$S_1(1\tensor \tf_1 (1^{\tensor x})) = S_1(1 \tensor 2 \tensor 1^{\tensor (x-1)})
= 1 \tensor 2\tensor 2^{\tensor (x-1)} = 1 \tensor S_1 (1^{\tensor x}),$$
and the last equality follows from
$$S_2(1 \tensor S_1(1^{\tensor x}) \tensor 3) =
S_2(1\tensor 2^{\tensor x} \tensor 3)=1\tensor 3^{\tensor (x-1)}\tensor2\tensor 3
=1 \tensor \te_2S_2 S_1 (1^{\tensor x}) \tensor 3.$$
Since $\te_{\ol 2} b_0 = 0$ by the assumption, $\teone S_2 S_1 b_0=0$,
and \eqref{eq:s2s1} implies
$$\te_{\ol 1} \te_{2} S_2 S_1 b_1 = 0.$$
On the other hand,  $\tf_1(b_1\otimes 3\otimes b_2)=\tf_1\tf_2b\not=0$ implies
$\tf_1b_1\not=0$.
Hence $b_1$ contains $1$, and
$\te_{\ol 1} b_1 =0 $ implies that
$b_1 = b_3 \tensor 1 \tensor b_4$ where
$b_4$ contains neither $1$ nor $2$.
Since $b_3$ is a $\gl(3)$-highest weight vector,
we have $S_1 b_1  = S_1 b_3 \tensor 2 \tensor b_4$.
Since $\te_2 S_1 b_1 =0$ by \eqref{eq:wh}, we have $\te_2 S_1 b_3 =0$.
Then we have
$\te_2 S_2 S_1 b_1 =b_5 \tensor 2 \tensor b_4$,
for some $b_5$
because
$\te_2S_2(2^{\otimes y} \otimes 2 \otimes 3^{\otimes z})
=\te_2 \bigl(3^{\otimes (y+1-z)} \otimes 2^{\otimes z} \otimes 3^{\otimes z} \bigr)
=3^{\otimes(y-z)} \otimes 2^{\otimes z}  \otimes 2 \otimes 3^{\otimes z}$.
This contradicts $\te_{\ol 1} \te_2 S_2 S_1 b_1 = 0$.
Hence we get the desired result $\teone b=0$.
\end{proof}
\end{lemma}

\begin{lemma} \label{le:e1barbneq0}
If $\varepsilon_1 b = 0$ and $\langle k_1, \wt(b) \rangle = \langle k_2, \wt(b) \rangle > 0$,
then $\te_{\ol 1} b \neq 0$.
\begin{proof}
Assume that $\te_{\ol 1} b = 0$.
Then $b=b_1 \tensor 1 \tensor b_2$ for some $b_1$ and $b_2$,
where $b_2$ contains neither $1$ nor $2$.
Since $\varepsilon_1(b_1)=0$, we have
$$\langle k_1 , \wt(b) \rangle = 1+ \langle k_1 , \wt(b_1) \rangle \geq 1+\langle k_2 , \wt(b_1) \rangle = \langle k_2 , \wt(b) \rangle +1,$$
which is a contradiction.
\end{proof}
\end{lemma}

\begin{prop} \label{prop:strict partition}
If $b$ is a highest weight vector in $\B^{\tensor N}$, then $\wt(b)$ is a strict partition.
\end{prop}
\begin{proof} 
Assuming that $\langle k_i, \wt(b) \rangle = \langle k_{i+1}, \wt(b)
\rangle > 0$, we shall derive a contradiction.
Set $b' \seteq S_{w_i} b$. Since $w_i^{-1}(\alpha_1)=\alpha_i$,
\eqref{eq:wh} implies $\te_1 b'=0$. Hence
 Lemma~\ref{le:e1barbneq0} implies $\teone b'\not=0$, which is a contradiction.
\end{proof}

\begin{lemma} \label{le:A}
Let $b$  be a vector in $\B^{\tensor N}$.

{\rm (a)} If $\te_{\ol 1} b= \te_1 b=0$ and $\langle k_1, \wt(b)
\rangle \geq \langle k_2, \wt(b) \rangle +2$, then $\te_{\ol 1} (1
\tensor \tf_1 b) = 0$.

{\rm (b)} 
If $\te_{\ol 1} b= \te_1 b = \te_2 b =0$ and $\langle k_2,
\wt(b) \rangle > \langle k_3, \wt(b) \rangle$, then $\te_{\ol
1} (1 \tensor \tf_1 \tf_2 b) = 0$.
\end{lemma}
\begin{proof}

(a)   Since $\langle k_1, \wt(b) \rangle >0$ and $\te_{\ol 1} b=0$,
we can write $b=b_1 \tensor 1 \tensor b_2$ for some $b_1$ and $b_2$
such that $b_2$ contains neither $1$ nor $2$. Then we get
$$2 \leq \langle k_1, \wt(b) \rangle - \langle k_2, \wt(b) \rangle
= \langle k_1, \wt(b_1) \rangle - \langle k_2, \wt(b_1) \rangle + 1
=\varphi_1(b_1) - \varepsilon_1(b_1)+1.$$
Thus $\varphi_1(b_1) > 0 = \varepsilon_1(1)$ and hence $\tf_1 b= \tf_1 b_1 \tensor 1 \tensor b_2$.
It follows that $\teone(1 \tensor \tf_1 b)=0$.



\bigskip
\noi
(b) Since $\langle k_1, \wt(b) \rangle \geq \langle k_2, \wt(b)
\rangle
>0$ and $\te_{\ol 1} b=0$, we can write $b=b_1 \tensor 1 \tensor
b_2$ for some $b_1$ and $b_2$ such that $b_2$ contains neither $1$
nor $2$. It follows that $\varepsilon_2(b_1) = \varphi_2(b_2)=0$.
Observe that 
$$\varphi_2(b_1) = \langle k_2, \wt(b_1) \rangle - \langle k_3, \wt(b_1) \rangle
> \langle k_3, \wt(b_2) \rangle - \langle k_2, \wt(b_2) \rangle
                 = \varepsilon_2(b_2).$$
Hence we have $\tf_2 b = \tf_2 b_1 \tensor 1 \tensor b_2$.
Since $\varepsilon_1(\tf_2 b_1)=0$,
we deduce that
$$\varphi_1(\tf_2 b_1)= \langle k_1-k_2, \wt(\tf_2 b_1) \rangle
=\langle k_1-k_2, \wt(b_1) \rangle  + 1
=\varphi_1(b_1) +1 > 0= \varepsilon_1(1 \tensor b_2),$$
and hence
$$\tf_1 \tf_2 b = \tf_1 \tf_2 b_1 \tensor 1 \tensor b_2.$$
Therefore we have $\te_{\ol 1} (1 \tensor \tf_1\tf_2 b)=0$.
\end{proof}

\begin{prop} \label{pro:h.w. vectors 2}
If $b \in \B^{\tensor (N-1)}$ is a highest weight vector with
$\langle k_{j-1}, \wt(b) \rangle \geq \langle k_j, \wt(b) \rangle +
2$, then $b_0 = 1 \tensor \tf_1 \cdots \tf_{j-1} b$ is a highest
weight vector in $\B^{\tensor N}$.

\begin{proof}
We will show $\te_{\ol i} b_0 = 0$ for $i = 1,2,\ldots, n-1$.

\vskip 3mm
\noi
{\bf Case 1:} $i \geq j+1$.

By Lemma~\ref{le:swib0}, we have
$$S_{w_i} b_0 = 3 \tensor \tf_3 \cdots \tf_{j+1} S_{w_i} b.$$
Thus we obtain
$$\te_{\ol 1} S_{w_i} b_0 = 3 \tensor \tf_3 \cdots \tf_{j+1} \te_{\ol 1} S_{w_i} b = 0.$$

\medskip
\noi
{\bf Case 2:} $i=j$.

Since $S_{w_i}b_0 = 1 \tensor S_{w_i} b$, we have
$\te_{\ol 1} S_{w_i} b_0=0$.\\

\medskip
\noi
{\bf Case 3:} $i=j-1$.

We have
$$S_{w_i} b_0 = 1 \tensor \tf_1 S_{w_i} b$$
and
$$\langle k_1, \wt(S_{w_i}b) \rangle = \langle k_{j-1}, \wt(b) \rangle
\geq \langle k_j, \wt(b) \rangle +2= \langle k_2, \wt(S_{w_i}b)
\rangle +2.$$
By Lemma~\ref{le:A}(a), we obtain $\te_{\ol 1} S_{w_i} b_0 = 0$.  \\

\medskip
\noi
{\bf Case 4:} $i \leq j-2$.

Set $b' \seteq \tf_{i+2} \cdots \tf_{j-1} b$.
Here we understand $b'=b$ if $i=j-2$.
Then $b'$ is a $\gl(i+2)$-highest
weight vector and $\te_{\ol 1} b'=0$.
Because $\te_{\ol 1}$ commutes with $S_{u_i}$, we have $\te_{\ol 1} S_{u_i}b'=0$.
Since $u_i^{-1}(\alpha_1)$ and $u_i^{-1}(\alpha_2)$ are positive roots,
$S_{u_i} b'$ is a $\gl(3)$-highest weight vector by \eqref{eq:wh}.

By Lemma~\ref{le:swib0}, we have
$$S_{u_i} b_0 = 1 \tensor \tf_1 \tf_2 S_{u_i} b'.$$
Observe that
$$
\begin{aligned}
\langle k_2, \wt(S_{u_i}b') \rangle - \langle k_3, \wt(S_{u_i}b') \rangle
 &= \langle k_{i+1}, \wt(b') \rangle - \langle k_{i+2}, \wt(b') \rangle \\
 &= \langle k_{i+1}-k_{i+2} , \wt(b)-\eps_{i+2} + \eps_j \rangle \\
 &= \langle k_{i+1}-k_{i+2}, \wt(b) \rangle + 1 - \delta_{j, i+2} \\
 & \geq 1.
\end{aligned}
 $$
 By Lemma~\ref{le:A}(b), we get $\te_{\ol 1} S_{u_i} b_0 = 0$. Since
$S_{u_i}=S_{z_i} S_{w_i}$ and $\te_{\ol 1}$ commutes with $S_{z_i}$,
we obtain
$\te_{\ol 1} S_{w_i} b_0 =0$.
\end{proof}
\end{prop}

\begin{theorem} \label{th:char.h.w}
Assume that $b$ is a $\gl(n)$-highest
weight vector in $\B^{\tensor (N-1)}$
and $b_0\seteq1 \tensor \tf_1 \cdots \tf_{j-1} b$ is a $\gl(n)$-highest
weight vector in $\B^{\tensor N}$. Then $b_0$ is a highest
weight vector if and only if $b$ is a highest weight vector and
$\wt(b_0) = \wt(b)+\eps_j$ is a strict partition.

\begin{proof}
Note that $\wt(b)$ and $\wt(b_0)$ are partitions.

If $b_0$ is a highest weight vector, then by Theorem~\ref{th:h.w.
vectors} and Proposition~\ref{prop:strict partition}, $b$ is a
highest weight vector and
 $\wt(b_0)$ is a strict partition.

Conversely, if $b$ is a highest weight vector so that $\wt(b)$ is a strict partition and $\wt(b) + \eps_j$ is still a strict partition,
then we have
 $\langle k_{j-1} - k_j, \wt(b) \rangle \geq 2 $
and hence, by Proposition~\ref{pro:h.w. vectors 2}, $b_0$ is a
highest weight vector.
\end{proof}
\end{theorem}

\vskip 1cm

\section{Existence and uniqueness}

In this section, we state and prove the main result of our paper:
the existence and uniqueness theorem for crystal bases. We first
prove several lemmas that are needed in the proof of our main
theorem.

We set
\eq
\tki=
q^{k_i-1} k_{\ol i}\quad\text{ for all $i=1, \ldots, n$.}
\eneq

\begin{lemma} \label{le:invarianct under tildeki}
Let $M$ be a $\Uq$-module in $\Oint$.
\bna
\item
For $\mu\in \wt(M)$ and $i\in\{1,\ldots,n-1\}$
such that $\mu+\alpha_i\not\in\wt(M)$, we have
$$\tk_{\ol{i+1}}=S_i\circ\tk_{\ol i}\circ S_i\quad\text{as endomorphisms of $M_\mu$,}$$
where $S_i$ is defined in {\rm Remark~\ref{rem:teibar}}.
\item Assume that $\la\in\wt(M)$ satisfies $\la+\alpha_i\not\in \wt(M)$
for all $i=1,\ldots,n-1$.
If $(L, B, l_{B})$ is a
crystal basis of $M$, then $L_\la$ is invariant under
$\tki$ for all $i=1,\ldots,n$.
\ee
\end{lemma}
\begin{proof}
(a) Set $\ell\seteq\lan h_i,\mu\ran\ge0$.
Then $S_i\cl M_\mu\isoto M_{s_i\mu}$
is given by $f_i^{(\ell)}$, and
its inverse is given by $e_i^{(\ell)}$.
%
Note that $e_iM_\mu=0$.

From the defining relation it follows that
\begin{equation*}
\begin{array}{lll}
e_{\ol i} f_i - f_i e_{\ol i} &= e_{\ol i} q^{-k_i} q^k_i f_i - f_i e_{\ol i} q^{-k_i} q^{k_i}\\
                              &= (k_{\ol i} e_i - q e_i k_{\ol i}) q^{k_i} f_i - f_i (k_{\ol i} e_i - q e_i k_{\ol i}) q^{k_i} \\
                              &= (k_{\ol i} e_i - q e_i k_{\ol i}) f_i q^{k_i -1}  - f_i (k_{\ol i} e_i - q e_i k_{\ol i}) q^{k_i} \\
                              &= k_{\ol i} e_i f_i q^{k_i -1} - q e_i k_{\ol i} f_i q^{k_i -1} - f_i k_{\ol i} e_i q^{k_i} + q f_ie_i k_{\ol i} q^{k_i}\\
                              &= k_{\ol i}\left(f_i e_i + \dfrac{q^{h_i}-q^{-h_i}}{q - q^{-1}} \right)q^{k_i -1}
                              - q e_i k_{\ol i} f_i q^{k_i -1} - f_i k_{\ol i} e_i q^{k_i} + q f_ie_i k_{\ol i} q^{k_i}.

\end{array}
\end{equation*}
Thus on $M_{\mu}$, we have
\begin{equation*}
\begin{array}{lll}
e_{\ol i} f_i - f_i e_{\ol i} &= k_{\ol i}
\dfrac{q^{h_i}-q^{-h_i}}{q - q^{-1}} q^{k_i -1} -  e_i k_{\ol i} f_i
q^{k_i},
\end{array}
\end{equation*}
which yields
\begin{equation} \label{eq:k i+1 bar}
\begin{array}{lll}
k_{\ol{i+1}} &= q^{-h_i} k_{\ol i} -(e_{\ol i} f_i - f_i e_{\ol i}) q^{-k_i} \\
             &= q^{-h_i} k_{\ol i} - \left(k_{\ol i} \dfrac{q^{h_i}-q^{-h_i}}{q - q^{-1}} q^{k_i -1} -  e_i k_{\ol i} f_i q^{k_i} \right) q^{-k_i} \\
             &= q^{-h_i} k_{\ol i} - k_{\ol i} \dfrac{q^{h_i}-q^{-h_i}}{q - q^{-1}} q^{-1} +  e_i k_{\ol i} f_i  \\
             &= -k_{\ol i} \left(  \dfrac{q^{h_i -1}-q^{-h_i+1}}{q - q^{-1}} \right) +  e_i k_{\ol i} f_i  \\
             &= -[\ell-1] k_{\ol i} +  e_i k_{\ol i} f_i.  \\
\end{array}
\end{equation}

On the other hand, we have, similarly to \eqref{eq:k1serre},
\begin{equation*}
\begin{array}{lll}
e_i ^{(2)} k_{\ol i} = e_i k_{\ol i} e_i- k_{\ol i} e_i^{(2)}. \\
\end{array}
\end{equation*}
By induction on $s$, we obtain
\begin{equation*}
\begin{array}{lll}
e_i ^{(s)} k_{\ol i} = e_i k_{\ol i} e_i^{(s-1)}-[s-1] k_{\ol i} e_i^{(s)} & (s \geq 1). \\
\end{array}
\end{equation*}
If $\ell >0$, we have on $M_{\mu}$
\begin{equation*}
\begin{array}{lll}
e_i^{(\ell)} q^{k_{i}-1} k_{\ol i} f_i^{(\ell)}
   &= q^{-\ell} e_i^{(\ell)} k_{\ol i} f_i^{(\ell)} q^{k_i -1} \\
   &= q^{-\ell} (e_i k_{\ol i} e_i^{(\ell-1)}-[\ell-1] k_{\ol i} e_i^{(\ell)}) f_i^{(\ell)} q^{k_i -1} \\
   &= q^{-\ell} (e_i k_{\ol i} f_i - [\ell-1]k_{\ol i}) q^{k_i -1} \\
   &= (e_i k_{\ol i} f_i - [\ell-1]k_{\ol i}) q^{k_{i+1} -1} \\
   &= k_{\ol{i+1}} q^{k_{i+1} -1} = \tk_{\ol {i+1}}.
\end{array}
\end{equation*}

\noi
If $\ell=0$, then $f_i M_{\mu}=0$,  and hence
\eqref{eq:k i+1 bar} implies $k_{\ol{i+1}} = k_{\ol i}$.
Therefore we have
$\tk_{\ol{i+1}}=k_{\ol{i+1}} q^{k_{i+1}-1}= k_{\ol i}q^{k_{i}-1}=\tk_{\ol{i}}$ on $M_{\mu}$.
In the both cases, we have $\tk_{\ol{i+1}} = S_i^{-1} \tki S_i$ on $M_{\mu}$.

\bigskip
\noi
(b)\quad
Let $M'=\uqqn M_\la \subset M$, and let $L'=L \cap M'$.
Set $\mu_j \seteq s_{j}\cdots s_{i-1} \la $ for $j=1,\ldots, i$.
Then $\langle h_{j}, \mu_{j+1} \rangle \ge0$, $s_j\mu_{j+1}=\mu_j$, and
$\mu_{j+1} + \alpha_j \notin \wt(M')$.
{}From (a) it follows that
$\tk_{\ol{j+1}}\vert_{M'_{\mu_{j+1}}}=S_j\circ\tk_{\ol j}\circ S_{j}\vert_{M'_{\mu_{j+1}}}$.
Hence, if $L'_{\mu_j}$ is stable under  $\tk_{\ol{j}}$,
then $L'_{\mu_{j+1}}$ is stable under  $\tk_{\ol{j+1}}$.
Since $L'_{\mu_1}$ is stable under $\tkone$,
$L'_{\mu_j}$ is stable under $\tk_{\ol{j}}$ for
all $j=1,\ldots, i$ by induction. 
In particular, $L_{\la}=L'_\la$  is stable under $\tki$.
\end{proof}

\vskip 3mm

\begin{lemma} \label{le:uniqueness of crystal lattice}
  Let $M$ be a $\uqqn$-module in $\Oint$,
and $\la\in P^{\ge0}$.
Let $(L, B, l_B)$ be a crystal basis of $M$ such that
any connected component of $B$ intersects with $B_\la$.
Let $L'$ be an $\A$-submodule of $M$
with the weight space decomposition $L'=\soplus_{\mu \in P^{\geq 0}} (L' \cap M_{\mu}) $,
which is stable under  $\tei$, $\tfi$ $(i=1,\ldots,n-1,\ol1)$.
Then
\bna

\item  $L'_{\lambda} \subset L_{\lambda}$ implies $L' \subset
L$,

\item  $L'_{\lambda} \supseteq L_{\lambda}$ implies $L' \supseteq L$.
\end{enumerate}

\begin{proof}
(a)
Assume that $L'_{\lambda} \subset L_{\lambda}$.
Set $S \seteq (L \cap q L') / (qL \cap qL') $.
Then $S \subset L/qL$ and $S$ is stable under $\tei$, $\tfi$
($i=1,\ldots,n-1,\ol1$).
Note that
$$S_{\lambda} = S \cap (L/qL)_{\lambda} = (L_{\lambda} \cap q L'_{\lambda}) / (q L_{\lambda} \cap q L'_{\lambda}) = 0.$$
For each $b \in B$, let $P_b \cl L / qL \twoheadrightarrow l_b$ be the
canonical projection.
Since $S_{\lambda} = 0$, we have
$P_{b} (S) = 0$ for any $b \in B_\la$. If $\tei b \neq 0$ for some $i=1,\ldots, n-1,
\overline{1}$, then $\tei\circ P_b = P_{\tei b}\circ \tei$ implies
$\tei P_b (S) = P_{\tei b} \tei (S) \subset P_{\tei b} (S)$.
Therefore, if $P_{\tei b} (S) = 0$, then $P_b (S) = 0$.
The same property holds for $\tfi$.

Since any $b \in B$ can be connected with
an element of weight $\la$ by a sequence of
operators in  $\tei$, $\tfi$ ($i=1,\ldots,n-1,\ol1$), we have
$P_b(S)=0$ for all $b \in B$. It follows that $S=0$ and hence $L
\cap q L' \subset q L$.

Since
$$L ' \cap q^{-m} L \subset q^{-(m-1)} (L' \cap q^{-1} L ) \subset q^{-(m-1)} L$$
for all $m\ge1$, we have $L ' \cap q^{-m} L \subset L' \cap q^{-(m-1)} L$.
Hence we obtain $L' \cap q^{-m} L \subset L$.
It follows that $L' \subset L$ as desired.

\bigskip
\noindent
(b) Assume that $L'_{\lambda} \supset L_{\lambda}$.
Set $S \seteq (L' \cap L) / (L' \cap qL)$.
Then $S \subset L/qL$ and $S$ is stable under $\tei$, $\tfi$.
Note that $l_{b}\subset S$ for any $b \in B_\la$.
If $\tei b  \neq 0$ and $l_b \subset S$,
 then $l_{\tei b} = \tei l_b \subset \tei S \subset S$.

The same is true for $\tfi$. Thus we have
$L/qL = \soplus_{b \in B} l_b \subset S$. By Nakayama's lemma, we
have $L' \cap L = L$.
\end{proof}
\end{lemma}

\begin{lemma} \label{le:crystal bases for irreducible summands}
Let $M$ be a highest weight $\uqqn$-module with highest weight $\la \in \La^+$
in the category $\Oint$.
Suppose that $M$ has a crystal basis $(L, B, l_B)$ such that
$B_\la = \{b_\la \}$ and $B$ is connected.
Let $L_\la = \soplus_{j=1}^s E_j$ be a decomposition into indecomposable modules over
$\A[C_1,\ldots, C_r]$ \ro see {\rm Remark~\ref{re:Clifford algebras}}\rf,
where $r=\ell(\la)$, and let
$$
\begin{aligned}
M_j \seteq \uqqn E_j, & \quad L_j \seteq M_j \cap L & \text{and} \quad l^{j}_b \seteq l_b \cap \big(L_j / qL_j \big).
\end{aligned}
$$
Then we have
\bna \item $M_j$ is irreducible over $\uqqn$, \item
$M=\soplus_{j=1}^s M_j$, $L = \soplus_{j=1}^s L_j$ and $l_b =
\soplus_{j=1}^s l^j_b$, \item $(L_j, B, (l^{j}_b)_{b \in B})$ is
a crystal basis of $M$. \ee
\begin{proof}
By Remark~\ref{re:Clifford algebras},
we see that
$(M_j)_\la \simeq\F \otimes_\A E_j$
is an irreducible module over $\F[C_1, \ldots, C_r]$ for each $j = 1,2, \ldots, s$.
Hence, Proposition~\ref{prop:uniqueness of lattice of highest weight space of an irreducible module} (a) implies
that $M_j$ is irreducible over $\uqqn$ and $M = \soplus_{j=1}^s M_j$.
Note that
$$L_j/ q L_j \subset L / qL \quad (j=1,2, \ldots, s).$$

Since $\soplus_{j=1}^s(L_j)_\la =\soplus_{j=1}^s \big(M_j \cap L_\la \big)=\soplus_{j=1}^s E_j=L_\la$,
we have
$$l_{b_{\la}} = L_\la / q L_\la = \soplus_{j=1}^s \big( (L_j)_\la / q(L_j)_\la \big)
=\soplus_{j=1}^s l^j_{b_{\la}}.$$
Consider $b_1, b_2 \in B$ such that $b_2=\tei b_1$ (equivalently,
$b_1=\tfi b_2$) for some $i =1,2,\ldots, n-1, \ol 1$.
Then we have injective maps
$$\tei|_{l^{j}_{b_1}} \cl l^{j}_{b_1} \mono l^{j}_{b_2}, \quad \tfi|_{l^{j}_{b_2}}
\cl l^{j}_{b_2} \mono l^{j}_{b_1}.$$
Hence comparing their dimensions, we conclude that
$$\tei \cl {l^{j}_{b_1}} \isoto l^{j}_{b_2}\quad\text{and}\quad
\tfi \cl {l^{j}_{b_2}} \isoto l^{j}_{b_1}
\quad\text{for all $j=1,2,\ldots, s$.}$$

Therefore $l_{b_1}=\soplus_{j=1}^s l^{j}_{b_1}$ if and only if
$l_{b_2}=\soplus_{j=1}^s l^{j}_{b_2}$.
Since $B$ is connected,
$\soplus_{j=1}^s l^{j}_{b} =l_{b}$
for all $b \in B$.

\noi
Since
$$L / q L = \soplus_{b \in B} l_b =
\soplus_{j=1}^s \soplus_{b \in B}  l^{j}_b
\subset \soplus_{j=1}^s L_j /qL_j,$$
Nakayama's lemma implies that
$L=\soplus_{j=1}^s L_j$, and $(L_j, B, (l^{j}_b)_{b \in B})$
is a crystal basis of $M_j$. 
\end{proof}
\end{lemma}

\begin{lemma} \label{le:isomorphic crystals}
Let $M$ be a $\Uq$-module in the category $\Oint$ and
let $(L_1, B_1, l^1_{B_1})$, $(L_2, B_2, l^2_{B_2})$ be
two crystal bases of $M$ such that $L_1=L_2$. If $B_1$ is a
connected abstract $\qn$-crystal and there exist $b_1 \in B_1$, $b_2
\in B_2$ such that $l^1_{b_1}=l^2_{b_2}$, then there exists a bijection
$\vphi \cl B_1 \to B_2$
 which commutes with the Kashiwara operators and $l^1_b=l^2_{\vphi(b)}$
for all $b \in B_1$.

\begin{proof} 
Let us set $S=\set{b\in B_1}{\text{there exists $b'\in B_2$ such that
$l^1_b=l^2_{b'}$}}$.
Then it is easy to see that it is stable under the Kashiwara operators
and it contains $b_1$.
Hence $S$ coincides with $B_1$.
Therefore for every $b\in B_1$,
there exists a $b'\in B_2$ such that $l^1_b=l^2_{b'}$.
Such a $b'$ is unique and we can define
$\vphi$ by $\vphi(b)=b'$.
Since $L_1/qL_1=\soplus\nolimits_{b\in B_1}l^1_b=\soplus\nolimits_{b\in B_2}l^2_b$,
$\vphi\cl B_1\to B_2$ is bijective.
\end{proof}
\end{lemma}

\begin{lemma} \label{le:existence of crystal base for a lattice}
Let $\la \in \Lambda^+$ and assume that $V(\la)$ has a crystal basis $(L_0,B_0,l_{B_0})$
such that $B_0$ is connected and $(B_0)_\la =\{b_\la\}$.
 Let $M \in \Oint$ be a highest weight $\Uq$-module with highest weight
$\la \in \Lambda^+$.
If $E$ is a free $\A$-submodule of $M_{\la}$, which is stable under $\tki$
for $i=1,2,\ldots, n$ and generates $M_{\la}$ over $\F$,
then there exists a unique crystal basis $(L,B,l_B)$ such that
\bna
\item $L_{\la}=E$,

\item $B \simeq B_0$ as an abstract $\qn$-crystal,
\end{enumerate}
\end{lemma}
\begin{proof}
By Lemma~\ref{le:invarianct under tildeki}
and Proposition~\ref{prop:uniqueness of lattice of highest weight
space of an irreducible module}, there exists a finitely generated
free $\A$-module $K$ such that $M \simeq K \tensor_\A V(\la)$ and $E
\simeq K \tensor_\A (L_0)_\la$. Then $(K \tensor_\A (L_0), B_0, (K
\tensor l_b)_{b \in B_0})$ is a crystal basis for $M$. The
uniqueness follows from Lemma~\ref{le:uniqueness of crystal lattice}
and Lemma~\ref{le:isomorphic crystals}.
\end{proof}

\vskip 1ex
For a weight $\la = \la_1 \epsilon_1 + \cdots + \la_n
\epsilon_n \in P$, define $|\la| =\sum_{i=1}^{n} \la_i$.
Now we are ready to state our main theorem.

\Th \label{th:main theorem} \hfill
\bna
\item  Let $M$ be an irreducible highest weight $\Uq$-module with highest weight
$\la \in \La^{+}$. Then there exists a crystal basis $(L, B, l_{B})$
of $M$ such that

\bni

\item $B_{\la} = \{b_{\la} \}$,

\item $B$ is connected.

\end{enumerate}
Moreover, such a crystal basis is unique up to an automorphism of
$M$. In particular, $B$ depends only on $\la$ as an abstract
$\qn$-crystal and we write $B=B(\la)$.

\item The $\qn$-crystal $B(\la)$ has a unique highest weight vector
$b_{\la}$. 

\item A vector $b \in \B \otimes B(\la)$ is a highest weight vector
if and only if
$$b = 1 \otimes \tf_{1} \cdots \tf_{j-1} b_{\la}$$
for some $j$ such that $\la + \epsilon_j$ is a strict partition.

\item Let $M$ be a finite-dimensional highest weight $\Uq$-module
with highest weight
$\la \in \La^{+}$. 
Assume that $M$ has a crystal basis $(L,B(\la), l_{B(\la)})$
such that $L_{\la} / q L_{\la} = l_{b_{\la}}$.
Then we have

\bni

\item $\V \otimes M = \soplus_{\la + \epsilon_j : \text{strict}}M_ j,$
where $M_j$ is a highest weight $\Uq$-module with highest
weight $\la + \epsilon_j$ and $\dim (M_{j})_{\la + \epsilon_j} = 2\dim M_{\la}$,

\item 
if we set $L_{j} = (\mathbf{L} \otimes L) \cap M_{j}$ and
$B_j=\set{b\in\B \otimes B(\la)}{l_b\subset L_j/qL_j}$, then we have
$\B \otimes B(\la)=\coprod\limits_{\la+ \epsilon_j: \text{strict}} B_{j}$
and $L_j/qL_j = \soplus_{b \in B_j} l_{b}$,
\item $M_j$ has a crystal basis $(L_j, B_j, l_{B_j})$,

\item $B_{j}\simeq B(\la + \epsilon_j)$ as an abstract $\qn$-crystal.

\end{enumerate}
\end{enumerate}
 \enth

\begin{proof}
We shall argue by induction on $|\la|$.

For a positive integer $k$, we denote by ${\rm (a)}_{k}$,
${\rm(b)}_{k}$, ${\rm(c)}_{k}$ and ${\rm(d)}_{k}$ the assertions
(a), (b), (c) and (d) for $\la$ with $|\la| = k$,  respectively.

It is straightforward to check ${\rm (a)}_{1}$ and ${\rm (b)}_{1}$.
Assuming the assertions ${\rm (a)}_{k}$,  ${\rm (b)}_{k}$
for $k \leq N$ and the assertions ${\rm (c)}_{k}$, ${\rm (d)}_{k}$ for $k < N $,
let us show ${\rm (a)}_{N+1}$,  ${\rm (b)}_{N+1}$,
${\rm (c)}_{N}$ and ${\rm (d)}_{N}$.

\vs{2ex}
\noindent
\textbf{Step 1} :
We shall prove ${\rm (c)}_N$.
Let $\la$ be a strict partition with $|\la|=N$.
By choosing a sequence of strict partitions $\eps_1=\la_1, \
\la_2, \ \la_3, \ldots, \la_N=\la $ such that
  $\la_{k+1} = \la_k + \eps_{j_k}$ for some $j_k$ and applying ${\rm (d)}_{k}$ on each $\la_k$ successively for $k < N$,
 we can embed $B(\la)$ into $\B^{\tensor N}$.
It follows that $\B \otimes B(\lambda) \subset \B^{\tensor
(N+1)}$.
By ${\rm (b)}_N$, we know that there exists a unique
highest weight vector, say $b_\la$, in $B(\la)$.
By Theorem~\ref{th:char.h.w}, an element $b \in \B \otimes B(\la)$ is a highest weight vector if and only if
$$b = 1 \otimes \tf_{1} \cdots \tf_{j-1} b_{\la}$$
for some $j$ such that $\la+\eps_j$ is a strict partition. So
${\rm (c)}_{N}$ holds.

\vs{2ex}
\noindent
\textbf{Step 2} : We shall show that ${\rm(d)}_N$ holds except (iv).
Let $M$ be a finite-dimensional highest weight module with highest
weight $\la \in \La^{+}$ with $|\la|=N$ and let $(L, B(\la),
l_{B(\la)})$ be a crystal basis of $M$. By Theorem
\ref{th:decomposition}, we have a decomposition
$\V \otimes M = \soplus_{\la +\epsilon_j : \text{strict}} M_ j$, where $M_j$ is a
highest weight $\Uq$-module with highest weight
$\la + \epsilon_j$ and $\dim (M_{j})_{\la + \epsilon_j} = 2 \dim
M_{\la}$.

By Theorem~\ref{th2:tensor product}, $\V
\otimes M$ admits a crystal basis $(\tilde{L},
\ \B \otimes B(\la), \ l_{\B \otimes B(\la)})$
where $\tilde{L}\seteq \mathbf{L} \otimes L$.
Set $L_j \seteq M_j \cap \tilde{L}$.
Note that
$$\F \otimes_\A L_j \isoto M_j \quad \text{and} \quad
L_j = \soplus_{\mu \in P^{\geq 0}} L_j \cap (M_j)_{\mu}. $$

Then we have
$$L_j / q L_j \subset \tilde L / q \tilde L = \soplus_{b \in \B \otimes B(\la)} l_b.$$
Since $\tei (M_j)_{\la+\epsilon_j} = \tilde e_{\ol i} (M_j)_{\la+\epsilon_j}=0 $
for any $i=1,2, \ldots, n-1$ (see Remark~\ref{rem:teibar}),
we have as subspaces of $\tL/q\tL$
$$(L_j)_{\la+\epsilon_j } / q (L_j)_{\la+\epsilon_j } \subset
 \Big( \bigcap_{i=1}^{n-1} \Ker \tei  \bigcap \bigcap_{i=1}^{n-1} \Ker \tilde e_{\ol i} \Big)_{\lambda+\epsilon_j} =
 \soplus _{\stackrel{\wt(b)=\la+\epsilon_j,}{\tei b = \tilde e_{\ol i} b =0 }} l_b = l_{b_{j}},$$
 where $b_j = 1 \otimes \tilde f_1 \cdots \tilde f_{j-1} b_{\la}$ in $\B \otimes B(\la)$.
Here, the last equality follows from ${\rm(c)}_{N}$.

Because ${\rm rank}_\A(L_j)_\mu = \dim_\F (M_j)_\mu$ for any $\mu \in \wt(M_j)$, we have
$$
\begin{aligned}
\dim_{\C} \big((L_j)_{\la+\epsilon_j } / q (L_j)_{\la+\epsilon_j } \big)=& {\rm rank}_\A(L_j)_{\lambda+\epsilon_j}
=\dim_\F (M_j)_{\lambda+\epsilon_j} \\
=&2 \dim_{\F} M_{\la} = \dim_{\C} l_{b_j}
\end{aligned}$$
and hence
$$(L_j / q L_j)_{\la + \epsilon_j}=(L_j)_{\la+\epsilon_j } / q (L_j)_{\la+\epsilon_j } = l_{b_j}.$$
Let $B_j$ be the connected component containing $b_j$ in $\B \otimes B(\la)$.
By ${\rm(c)}_N$ and Lemma~\ref{le:existence of
h.w. vectors}, we obtain $\bigcup_j B_j = \B \otimes B(\la)$.
Since $L_j$ is stable under $\tei, \tfi, \teone$ and $\tfone$, we have
$$\soplus_{b \in B_j} l_b  \subset L_j / q L_j. $$
It follows that \eq \nonumber \tilde L / q \tilde L = \soplus_{b
\in \B \otimes B(\la)} l_b =\soplus_{b \in \bigcup_j B_j} l_b
\subset \sum_j (L_j/ q L_j).
\eneq
By Nakayama's Lemma, we get \eq \label{eq:nakayama lemma} \tilde L
= \sum_j L_j. \eneq
Since $\sum_j L_j = \soplus\nolimits_j L_j$, we
obtain $\tilde L = \soplus\nolimits_j L_j$ and
\eq \nonumber
\soplus_{b \in \bigcup_j B_j} l_b = \tilde L / q \tilde L \simeq
\soplus_j (L_j/ q L_j) \supseteq  \soplus_j \soplus_{b \in B_j} l_{b_j}.
\eneq
Therefore, we obtain
$$ L_j / q L_j = \soplus_{b \in B_j} l_b\quad\text{and}
\quad\B \otimes B(\la) = \coprod_j B_j.$$
Thus $(L_j, \ B_j, \ l_{B_j} =(l_b)_{b \in B_j})$ is a crystal basis of $M_j$.

Note that each $B_j$ has a unique highest weight vector $b_j$ and
that $B_j$ is connected.

\vs{2ex}
\noindent
\textbf{Step 3} : 
We will show ${\rm (a)}_{N+1}$.
Since an irreducible highest weight module
is uniquely determined up to parity change, and since the crystal
structure dose not vary under the parity change functor, it is
enough to show that there exists an irreducible highest weight
module with a crystal basis which satisfies (i) and (ii) in ${\rm
(a)}$.

Let $\la$ be a strict partition with $|\la|=N+1$.
Choose a strict partition $\mu$ and $\ell=1,\ldots, n$
such that $\lambda=\mu +\epsilon_\ell$. By ${\rm (a)}_N$,
there exists an irreducible highest weight $\Uq$-module $M$
of highest weight $\mu$ which has a crystal basis $(L, B(\mu),
l_{B(\mu)})$. Then we have
$$\V \otimes M = \soplus_{\la + \epsilon_j : \text{strict}} M_j,$$
and each $M_j$ has a crystal basis as in {\bf Step 2}. Therefore
there exists a finite-dimensional highest weight $\Uq$-module $M$
with highest weight $\la$ which has a crystal basis $(L, B, l_{B})$
 such that $B$ is connected and $B_\la=\{b_0\}$.
Moreover we see that in {\bf Step 2}, $B$ has a unique highest weight vector.
By Lemma~\ref{le:crystal bases for irreducible summands}, we conclude
that each irreducible summand of $M$ admits a crystal basis
with the abstract crystal $B$ which satisfies (i) and (ii) in ${\rm (a)}$
and $B$ has a unique highest weight vector.

The uniqueness in ${\rm (a)}_{N+1}$ immediately follows from Lemma~\ref{le:invarianct under tildeki}, Lemma~\ref{le:isomorphic crystals} and Proposition~\ref{prop:uniqueness of lattice of highest weight space of an irreducible module}.
The property ${\rm (b)}_{N+1}$ is obvious.
The remaining (iv) in ${\rm (d)}_{N}$ follows
from Lemma~\ref{le:existence of crystal base for a lattice}.
\end{proof}

\begin{corollary} \hfill

\bna
\item Every $U_q(\mathfrak{q}(n))$-module in the category
$\Oint$ has a crystal basis.
\item If $M$ is a $\Uq$-module in the category
$\Oint$ and $(L,B,l_B)$ is a crystal basis of
$M$, then there exist decompositions $M=\soplus_{a\in A}M_a$ as a
$U_q(\mathfrak{q}(n))$-module, $L=\soplus_{a\in A}L_a$ as an
$\A$-module, $B=\coprod_{a\in A}B_a$ as a $\qn$-crystal,
parametrized by a set $A$ such that the following
conditions are satisfied for any $a\in A$ :
\bni
\item
$M_a$ is a highest weight module with
highest weight $\la_a$ and $B_a \simeq B(\la_a)$ for some strict
partition $\la_a$,
\item
$L_a=L\cap M_a$,
$L_a/qL_a=\soplus_{b\in B_a}l_b$, \item $(L_a,B_a,l_{B_a})$ is a
crystal basis of $M_a$.
\ee
\ee
\begin{proof}
(a) Our assertion follows from the semisimplicity of the category
$\Oint$. Indeed, if $M = \soplus_\nu M_{\nu}$
is a decomposition of $M$ into irreducible $\uqqn$-modules, then
each $M_{\nu}$ is an irreducible highest weight module, and hence it
admits a crystal basis $(L_\nu, B_\nu, l_{B_\nu})$ by Theorem
\ref{th:main theorem}. Set
\begin{equation*}
\begin{array}{ccc}
L \seteq \soplus_{\nu} L_\nu, & B \seteq  \coprod_\nu B_\nu, & l_B
\seteq (l_{b})_{b \in B}.
\end{array}
\end{equation*}
Then $(L, B, l_B)$ is a crystal basis of $M$.

\bigskip\noi 
(b) Let $\la$ be a maximal element in $\wt(B)=\wt(M)$. Note that if
$\ell(\la)=r$ is odd, then we have the following commutative
diagram (see Remark~\ref{re:Clifford algebras} for notations):
$$
\xymatrix{
\Mod(\A) \ar[r]^(.34){\sim} \ar[d]^{\C \tensor_{\A/q\A} ( - )} & \SMod(\A[C_1,\ldots, C_r]) \ar[d]^{\C \tensor_{\A/q\A} ( - )}\\
\Mod(\C) \ar[r]^(.34){\sim}        & \SMod(\C[C_1,\ldots, C_r]),
}
$$
and if $r$ is even, then we have the following commutative diagram:
$$\xymatrix{
\SMod(\A) \ar[r]^(.36){\sim} \ar[d]^{\C \tensor_{\A/q\A} ( - )} & \SMod(\A[C_1,\ldots, C_r]) \ar[d]^{\C \tensor_{\A/q\A} ( - )}\\
\SMod(\C) \ar[r]^(.36){\sim}        & \SMod(\C[C_1,\ldots, C_r]).
}
$$
The horizontal arrows are given by
$$K \mapsto V \tensor_\C K$$
for each module $K$ in the left hand side,
where  $V$ denotes an irreducible supermodule over $\C[C_1,\ldots,C_r]$.

Let $M^{(\la)}\seteq\uqqn M_\la$ be the isotypic component of $M$
that is a highest weight module of highest weight $\la$.
Let $B_{\la} = \set{b^\nu}{\nu =1,2, \ldots, s}$.
Then we have $L_\la/qL_\la=\soplus\nolimits_{\nu=1}^sl_{b^\nu}$.
Hence one can find an $\A[C_1,\cdots C_r]$-submodule $E_\nu$
of $L_\la$ for each $\nu =1,2, \ldots, s$ such that
$$E_\nu / qE_\nu =l_{b^{\nu}} \quad \text{and} \quad L_\la=\soplus_{\nu =1}^s E_\nu.$$
Setting $M^\nu\seteq\uqqn E_\nu$, we have
$$
M^{(\la)} = \soplus_{\nu =1}^s M^\nu.
$$
By Lemma~\ref{le:existence of crystal base for a lattice}, $M^{\nu}$
has a crystal basis
$$(L(M^\nu), B(\la), (l^\nu_b)_{b \in B(\la)})$$
such that $L(M^\nu)_\la=E_\nu$.
Hence their direct sum
$\soplus_{\nu=1}^s(L(M^\nu), B(\la), (l^\nu_b)_{b \in B(\la)})$ is a crystal basis
of $M^{(\la)}$.
Set $L(M^{(\la)})\seteq M^{(\la)}\cap L$.
Since $L(M^{(\la)})_\la=L_\la=\sum_\nu E_\nu=(\sum_\nu L(M^\nu))_\la$,
Lemma~\ref{le:uniqueness of crystal lattice}
implies $L(M^{(\la)})=\soplus\nolimits_{\nu=1}^sL(M^{\nu})$.
In particular, we have $L(M^\nu)=L\cap M^\nu$, and we can regard
$L(M^\nu)/qL(M^{\nu})$ as a subspace of $L/qL$.

\noi
The set
$\set{b \in B(\la)}{ l^\nu_{b} = l_{b'} \quad \text{for some } b' \in B}$
is stable under the Kashiwara operators and contains $b_\la$,
and hence it coincides with $B(\la)$. Therefore
the map $\phi_\nu \cl B(\la) \to B$ given by $l^\nu_b=l_{\phi_\nu(b)}$ ($b\in B(\la)$)
is injective and commutes with the Kashiwara operators.
Its image $C_\nu$ is therefore the connected component of
$b_\nu$ and we obtain
$$L(M^\nu) / q L(M^\nu) = \soplus_{b \in C_{\nu}} l_{b}.$$
%
Write $B=B_1 \sqcup B_2$, where $B_1=\coprod_{\nu=1}^s C_\nu$.
Then $(L(M^{(\la)}), B_1, l_{B_1})$ is a crystal basis of
$M^{(\la)}$ and coincides with the direct sum of the crystal bases
$(L(M^\nu), B(\la), l^\nu_{B(\la)})$ of $M^\nu$.

Let $M = M^{(\la)} \oplus \widetilde M$ be the decomposition as a $\uqqn$-module,
and set $\widetilde L \seteq L \cap \widetilde M$.

Set $S \seteq q^{-1}L(M^{(\la)}) \cap \bl(q^{-1} \widetilde L + L(M^{(\la)})\br)$.
Then $S$ is invariant under the Kashiwara operators and $S_\la =
L(M^{(\la)})_\la$. Hence by Lemma~\ref{le:uniqueness of crystal
lattice}, we have $S = L(M^{(\la)})$, which implies $L(M^{(\la)})
\cap (\widetilde{L}+qL(M^{(\la)}))=q L(M^{(\la)})$.
Hence we obtain
\eq
&&\big(L(M^{(\la)}) / q L(M^{(\la)})\big)
\cap \big(\widetilde L / q\widetilde  L \big) =0
\quad\text{as a subspace of $L/qL$.}
\label{eq:LMtilde}
\eneq
By comparing dimensions, we have
$$L/q L =\big(L(M^{(\la)}) / q L(M^{(\la)}) \big) \oplus \big(\widetilde
L / q \widetilde L\big).$$ Therefore, by Nakayama's lemma, we obtain
\eq
L = L(M^{(\la)}) + \widetilde L = L(M^{(\la)}) \oplus \widetilde L.
\label{eq:LLL}
\eneq

Now, we shall show
\eq
&&\widetilde L / q \widetilde L = \soplus_{b \in B_2} l_{b}.
\eneq
For $b\in B$, let $P_b \cl L/qL \twoheadrightarrow l_b$ be the canonical
projection. Then, for $i=1,\ldots, n-1,\ol1$ satisfying $\te_i b \in B$,
we have a commutative diagram
$$\xymatrix{
L/qL\ar[r]^{{\te_i}}\ar[d]^{P_b}&L/qL\ar[d]^{P_{\te_ib}}\\
l_b\ar[r]^-\sim_-{\te_i}&\;l_{\te_ib}\;.
}
$$
Hence $P_{\te_ib}(\tL/q\tL)=0$ implies $P_b(\tL/q\tL)=0$.
Similarly, $P_{\tf_ib}(\tL/q\tL)=0$ implies
$P_b(\tL/q\tL)=0$.
Hence $S\seteq\{b\in B_1\,;\,P_b(\tL/q\tL)=0\}$
is  stable under the Kashiwara operators. 
Since $S_\la=B_\la$, we obtain $S=B_1$.
Hence $\tL/q\tL\subset \soplus_{b\in B_2}l_b$.
Then \eqref{eq:LLL} implies
the desired result $\tL/q\tL=\soplus_{b\in B_2}l_b$.

Therefore $(\widetilde L, B_2, (l_b)_{b \in B_2})$ is a crystal
basis of $\widetilde M$. Hence
the crystal basis $(L,B, l_B)$ of $M$
is the direct sum of a crystal basis of $\widetilde M$
and crystal bases of $M^\nu$ ($\nu=1,\ldots,s$).
Since $\dim \widetilde M < \dim M$,
our assertion follows by induction on $\dim M$. 
\end{proof}
\end{corollary}

\end{document}